\documentclass[3p,preprint]{elsarticle}
\usepackage{graphicx}
\usepackage{amssymb}
\usepackage{amsthm}
\usepackage{amsmath}
\usepackage{empheq}
\usepackage{mathtools}
\usepackage{subcaption}
\usepackage{lineno} 
\usepackage{enumerate}
\usepackage{multirow} 
\usepackage{caption} 
\usepackage{array} 
\usepackage{booktabs} 
\usepackage{longtable}
\usepackage{bm}
\usepackage{setspace}
\usepackage{float}
\usepackage{bbm}
\usepackage[hidelinks]{hyperref}
\usepackage[figuresright]{rotating}

\newcounter{thmcounter}[section]
\theoremstyle{definition}
\newtheorem{definition}[thmcounter]{Definition}

\newtheorem{remark}[thmcounter]{Remark}
\theoremstyle{plain}
\newtheorem{theorem}[thmcounter]{Theorem}

\newtheorem{lemma}[thmcounter]{Lemma}

\numberwithin{equation}{section}
\numberwithin{thmcounter}{section}

\allowdisplaybreaks[3]

\begin{document}
\begin{frontmatter}

\title{Variable-order fractional $1$-Laplacian diffusion equations for 
multiplicative noise removal}
 
\author{Yuhang Li}
\ead{mathlyh@stu.hit.edu.cn}
\author{Zhichang Guo}
\ead{mathgzc@hit.edu.cn}
\author{Jingfeng Shao\cormark[1]}
\ead{sjfmath@hit.edu.com}
\author{Yao Li}
\ead{yaoli0508@hit.edu.cn}
\author{Boying Wu}
\ead{mathwby@hit.edu.cn}
\cortext[1]{Corresponding author.}
\address{The Department of Mathematics, Harbin Institute of Technology, Harbin 150001, China}

\begin{abstract}
    In this paper, we study a class of fractional $1$-Laplacian 
    diffusion equations with variable orders, proposed 
    as a model for multiplicative noise removal. The existence and uniqueness of 
    the weak solution are proven. To overcome the difficulties in the approximation process,
    we place particular emphasis on studying the density properties
    of the variable-order fractional Sobolev spaces. Numerical experiments demonstrate that 
    our model exhibits favorable performance across the entire image. 
\end{abstract}


\end{frontmatter}

\section{Introduction}
We study the following variable-order fractional $1$-Laplacian diffusion equation
\begin{equation}
    \left\{
    \begin{array}{ll}
        \displaystyle \vspace{0.2em} \frac{\partial u}{\partial t} = \int_{\Omega} 
        \frac{k(x,y)}{|x-y|^{N+s(x,y)}} \dfrac{u(t,y)-u(t,x)}{\left|u(t,y)-u(t,x)\right|} 
        d y, & \textrm{in~~} (0,T) \times \Omega,  \\ 
        u(0,x) = f, & \textrm{on~~} \Omega,
    \end{array}
    \right. \label{eqn:main}
\end{equation}
where $\Omega \subset \mathbb{R}^N$ with $N \geq 2$ is an open 
set of class $C^{0,1}$ with bounded boundary, $T>0$, $f$ is a positive measurable function, 
$s: \mathbb{R}^N \times \mathbb{R}^N \to (0,1)$ is some symmetric, continuous function, and 
$k(x,y)$ is some nonnegative symmetric weight function with compact 
support. This type of equation can be applied to remove multiplicative noise in images, 
when $u$ represents the restored image and $f$ the noisy image.

Multiplicative noise affects images in many fields, including synthetic aperture radar (SAR), 
ultrasonic imaging, and optical coherence tomography. The observed image $f$ is the pointwise 
product of the original image $u$ and noise $\eta$. Recovering the original image $u$ 
from degraded input $f$ is the main task in multiplicative noise removal,  
which constitutes a critical challenge in the field of image denoising. 
In the past three decades, a variety of methods based on reaction-diffusion equations 
have been proposed for image denoising. The classical diffusion equations for a 
density function $u$ can be traced back to 
the fundamental conservation law $u_t + \operatorname{div}(\vec{J})=0$, 
where $\vec{J}$ is the diffusion flux defined by Fick's first law $\vec{J} = - c\nabla u$. 
Assuming that the diffusion coefficient $a$ depends on the distribution in a manner 
expressed as a function $c = c(u)$, the equation is transformed to a nonlinear 
diffusion equation
\begin{equation} \label{eqn:nonlinear-diffusion}
\frac{\partial u}{\partial t} 
= \operatorname{div}\left(c(u) \nabla u\right),
\end{equation}
taking $c(u)=|\nabla u|^{p-2}$ ($1 \leq p < \infty$), 
then \eqref{eqn:nonlinear-diffusion} becomes the well-known $p$-Laplacian equation
\begin{equation} \label{eqn:pl}
    \frac{\partial u}{\partial t} 
    = \operatorname{div}\left(|\nabla u|^{p-2} \nabla u\right), 
\end{equation}
which has been extensively investigated for image processing \cite{SONG2003}. 
If $p=2$, this is the heat 
equation, which is equivalent to using a Gaussian filter for image denoising, leading to  
an overly smooth restoration of the image. If $p=1$, \eqref{eqn:pl} is the total variation  
(TV) flow, which comes from the gradient decent scheme of the ROF model \cite{RUDIN1992259}. 
Other types of nonlinear diffusion flows, in addition to the TV flow, have also been shown 
to effectively remove additive noise, see \cite{PERONA1990629, CATTE1992182, YOU20001723, 
CHEN20061383, GUO2012958}.

To perform the task of the multiplicative noise removal, inspired by the ROF model, 
Aubert and Aujol proposed the following (AA) variational model \cite{AUBERT2008925} using 
maximum a posteriori estimation (MAP), 
\begin{equation} \label{eqn:vaa}
    E_{AA}(u)=\int_{\Omega} |Du| + \int_{\Omega} \left(\log u +\frac{f}{u}\right)dx,
\end{equation}
and proved the existence and local uniqueness of \eqref{eqn:vaa}. The 
associated evolution problem of \eqref{eqn:vaa} is
\begin{subequations} 
    \begin{empheq}[left=\empheqlbrace]{alignat=2}
        & \vspace{0.2em} \frac{\partial u}{\partial t} = 
        \operatorname{div}\left(\frac{\nabla u}{|\nabla u|}\right) + \lambda \frac{f-u}{u^2},\,
        &&\textrm{~~~in~~} (0,T) \times \Omega, \label{eqn:aa-1} \\ 
        & \left\langle \nabla u, \vec{n} \right \rangle = 0, 
        &&\textrm{~~~on~~} (0,T) \times \partial \Omega, \label{eqn:aa-2} \\
        & u(0,x) = f, &&\textrm{~~~on~~} \Omega, \label{eqn:aa-3} 
    \end{empheq}
\end{subequations}
with \eqref{eqn:aa-1} being interpretable as a source term 
added to the TV flow. For other important variational models for multiplicative noise 
removal, see \cite{SHI2008294, LI20101, JIN201162, DONG2013912373}. The 
authors of \cite{ZHOU2015249} proposed a framework for 
multiplicative noise removal based on nonlinear diffusion equations. They have found 
that incorporating a source term derived from the variational problem into 
the diffusion equation may impede noise removal. As a result, it is  
recommended to select the source term as $0$ and accomplish the task of removing 
multiplicative noise by designing the diffusion coefficient in terms of $u$. For more 
diffusion-based models for multiplicative noise removal without a source term, we refer 
to \cite{ZHOU2018443, SHAN2019763}. By appropriately selecting 
the diffusion coefficient, both the edge and homogeneous 
regions can be restored during the process of image denoising. However, the previously 
mentioned models solely rely on local image information, thus fall 
short in preserving the texture and repetitive structures in images.

To effectively handle textures and repetitive structures in algorithms 
based on PDEs, it is proposed to utilize nonlocal operators for defining
novel types of flows and functionals in image processing and related 
fields \cite{GILBOA20091005}. Let $\Omega \subset \mathbb{R}^N$, $x,y \in \Omega$, $u(x)$ be a real 
function $u : \Omega \to \mathbb{R}$, $v(x,y)$ be a 
vector field $\vec{v}:\Omega \times \Omega \to \mathbb{R}$, 
$w(x,y) \geq 0$ is a symmetric weight function. The nonlocal 
gradient $\nabla_{\mathrm{NL}}u: \Omega \times \Omega \to \mathbb{R}$ is 
defined as 
\[(\nabla_{\mathrm{NL}}u)(x,y) \coloneqq (u(y)-u(x)) \sqrt{w(x,y)},\]
and the nonlocal divergence $\operatorname{div}_{\mathrm{NL}}\vec{v}: \Omega \to \mathbb{R}$ is 
defined as 
\[(\operatorname{div}_{\mathrm{NL}}\vec{v})(x) \coloneqq 
\int_{\Omega} (v(x,y)-v(y,x)) \sqrt{w(x,y)} dy, \]
Similar to the usual nonlinear diffusion equation, 
let $c(u) : \Omega \times \Omega \to \mathbb{R}$ be a symmetric function, then 
a nonlinear nonlocal diffusion operator can be defined as
\[\mathcal{N}u(x)
=\frac{1}{2} \operatorname{div}_{\mathrm{NL}}\left(c(u)\nabla_{\mathrm{NL}}u\right)(x)
= \int_{\Omega} (c(u)(x,y))(u(y)-u(x))w(x,y) dy,\]
then a nonlinear nonlocal diffusion process can be written as
\begin{subequations}  
    \begin{empheq}[left=\empheqlbrace]{align}
        & \vspace{0.2em} \frac{\partial u}{\partial t} = \mathcal{N}u,\ \ \,
        \textrm{in~~} (0,T) \times \Omega, \label{eqn:nonlocal-diffusion-1} \\ 
        & u(0,x) = f,  \textrm{~~~on~~} \Omega. \label{eqn:nonlocal-diffusion-2} 
    \end{empheq}
\end{subequations}

In the context of image processing, the authors of \cite{BUADES2005490} suggested 
using the weight function $w(x,y)$ to measure nonlocal similarity between 
two patches, with further applications 
documented in \cite{KINDERMANN20051091, GILBOA2007595, GILBOA20091005}. Various nonlocal 
diffusion models can be obtained by appropriately selecting $c(u)$ and $w(x,y)$. 
Let $J:\mathbb{R}^N \to \mathbb{R}^N$ be a 
nonnegative continuous radial function with compact support, $J(0)>0$,  
and $\int_{\mathbb{R}^N} J(z) dz = 1$, for $c(u)(x,y)=|u(y)-u(x)|^{p-2}$, 
$w(x,y)=J(x-y)$, \eqref{eqn:nonlocal-diffusion-1}
becomes a nonlocal $p$-Laplacian evolution equation 
\begin{equation} \label{eqn:nlpl}
    \frac{\partial u}{\partial t} 
    = \int_{\Omega} J(x-y) \left|u(t,y)-u(t,x)\right|^{p-2}\left(u(t,y)-u(t,x)\right) dy, 
\end{equation}
which has been studied in \cite{ANDREU2008201, ANDREU20091815}. If the kernel $J$ is 
appropriately rescaled, the solutions of \eqref{eqn:nlpl} converges strongly 
in $L^p$ to the solution of \eqref{eqn:pl} with Neumann condition. 
For $c(u)(x,y)=|u(y)-u(x)|^{p-2}$, $w(x,y)=|x-y|^{-(N+sp)}$, 
\eqref{eqn:nonlocal-diffusion-1}
becomes a fractional $p$-Laplacian evolution equation 
\begin{equation} \label{eqn:fpl}
    \frac{\partial u}{\partial t} = \int_{\Omega} 
    \frac{1}{|x-y|^{N+sp}} \left|u(t,y)-u(t,x)\right|^{p-2}\left(u(t,y)-u(t,x)\right) dy,
\end{equation}
which has been studied in \cite{MAZON2016810}. For $s<1$ 
close to $1$, the solution of \eqref{eqn:fpl} converges  
strongly in $L^p$ to the solution of \eqref{eqn:pl} with Neumann condition 
when $s \to 1^-$. 
If $p=2$, this is the well-known   
fractional Laplacian evolution equation. The fractional 
Laplacian $(-\Delta)^s$ arises in a number of applications such as anomalous diffusion, 
Lévy process, Gaussian random fields, quantum mechanics, 
see \cite{BOUCHAUD1990127,MEERSCHAERT2012,LINDGREN2011423,LASKIN2000298} and 
the references therein. For applications 
of the fractional Laplacian in image denoising 
problems, see \cite{GATTO2015249, ANTIL2017661}. The authors 
of \cite{LIU20191739} proposed the use of the fractional $1$-Laplacian  
in reaction-diffusion systems and applied it for additive noise removal. 
Combining the fractional 1-Laplacian and AA model, the following model \cite{GAO20224837}
\[\left\{
    \begin{array}{ll}
        \displaystyle \vspace{0.2em} \frac{\partial u}{\partial t} = \int_{\Omega} 
        \frac{k(x,y)}{|x-y|^{N+s}} \dfrac{u(t,y)-u(t,x)}{\left|u(t,y)-u(t,x)\right|} 
        d y + \lambda \dfrac{f-u}{u^2}, & \textrm{in~~} (0,T) \times \Omega,  \\ 
        u(0,x) = f, & \textrm{on~~} \Omega,
    \end{array}
\right.\]
was proposed to remove multiplicative noise. This model is referred to as the F1P-AA 
model in this article. The function $k(x,y)$ in the F1P-AA model can be appropriately 
selected to enhance the model's self-adaptability. The previously mentioned nonlocal methods 
perform well in removing noise in homogeneous regions and preserving textures, however, 
they also have two opposite drawbacks: causing over-smoothness in low-contrast regions 
and leaving residual noise around edges \cite{SUTOUR20143506}.

In practical images, there are often both homogeneous regions and edges, as well as 
textures and repetitive structures. In order to make nonlocal methods perform well across 
the entire image, there are currently two approaches. The first approach is to combine 
local and nonlocal methods, as shown in \cite{SUTOUR20143506, DELON2019458, 
GARRIZ2020112, SHI2021103362}. The other approach is 
letting $s$ be spatially dependent in the fractional-order operator. A variational model 
based on variable-order $x \mapsto s(x)$ was proposed in \cite{ANTIL20192479}, which can 
be viewed as a preliminary attempt to apply variable-order fractional Laplacian for image 
denoising. The established variational model is built on the basis of the Stinga-Torrea extension, 
instead of directly defining $(-\Delta)^{s(x)}$, because there is no clear way 
to define $(-\Delta)^{s(x)}$, actually. The well-defined variable-order fractional 
Laplacian $(-\Delta)^{s(\cdot,\cdot)}$ (see \cite{XIANG2019190}) and the variable-order 
fractional $p$-Laplacian $(-\Delta)_p^{s(\cdot,\cdot)}$ (see \cite{BU2021}) have not been applied to image 
processing yet, and there are rare theoretical results on their evolution problems.

Inspired by the models mentioned above, we propose a new model \eqref{eqn:main}
for multiplicative noise removal. Compared with the F1P-AA model, our model offers two 
unique benefits: enhanced self-adaptability through the selection of spatially 
dependent $s(x,y)$ and thorough denoising due to the absence of a source term. 
Based on the numerical experiments, it has been demonstrated that our model 
exhibits favorable performance across the entire  
image. In comparison with the F1P-AA model, our model is more effective 
in terms of denoising effects.

The rest of this paper is organized as follows. We state some necessary preliminaries about 
variable-order fractional Sobolev spaces and 
the main result in Section \ref{2}. In section \ref{3}, we prove more basic results of the 
space $W^{s(\cdot,\cdot),p}(\Omega)$, especially density properties. The existence 
and uniqueness of the variable-order fractional $p$-Laplacian 
evolution equations when $1<p<2$ is proven in Section \ref{4}. 
In Section \ref{5}, we prove that problem \eqref{eqn:main} 
admits a unique weak solution. Some properties of the weak 
solutions of \eqref{eqn:main} are presented in Section \ref{6}. 
Finally, we show some numerical experiments in Section \ref{7} 
to demonstrate the effectiveness of denoising by our model.

\section{Mathematical Preliminaries and main result} \label{2}

In this section, we first state some necessary preliminaries of variable-order fractional Sobolev spaces 
that will be used below. We refer to \cite{XIANG2019190, BU2021} and the references therein 
for some detailed proof. Henceforth, we will always assume 
that $\Omega$ is an open set in $\mathbb{R}^N$, 
$1 \leq p < + \infty$, $s: \mathbb{R}^N \times \mathbb{R}^N \to (0,1)$ is a symmetric, 
continuous function satisfies
\[0<s^-= \inf_{(x,y) \in \mathbb{R}^N \times \mathbb{R}^N} s(x,y) \leq s(x,y)
\leq s^+= \sup_{(x,y) \in \mathbb{R}^N \times \mathbb{R}^N} s(x,y) < 1.\]
Denote by 
\[[u]_{W^{s(\cdot,\cdot),p}(\Omega)} = \left(\int_{\Omega} 
\int_{\Omega} \frac{\left|u(y)-u(x)\right|^p}{|x-y|^{N+s(x,y)p}} dxdy \right)^{\frac{1}{p}}\]
the Gagliardo seminorm with variable order of a measurable function $u$ %
in $\Omega$. We consider the variable-order fractional Sobolev spaces 
\[W^{s(\cdot,\cdot),p}(\Omega) = \left\{u \in L^p(\Omega) \ : \
[u]_{W^{s(\cdot,\cdot),p}(\Omega)} < + \infty \right\},\] 
which is a Banach space with respect to the norm 
\[\|u\|_{W^{s(\cdot,\cdot),p}(\Omega)} \coloneqq \|u\|_{L^p(\Omega)} 
+ [u]_{W^{s(\cdot,\cdot),p}(\Omega)}.\]

Now let $\Omega$ be an open set with bounded boundary. Define the 
operator $T: W^{s(\cdot,\cdot),p}(\Omega) \to L^p(\Omega) \times L^p(\Omega \times \Omega)$ by 
\[T(u) = \left(u(x), \frac{u(x)-u(y)}{|x-y|^{\frac{N}{p}+s(x,y)}}\right),\]
clearly $T$ is an isometry. Thus, $W^{s(\cdot,\cdot),p}(\Omega)$ is 
separable (see \cite[Proposition 3.20]{BREZIS2011}), moreover, 
if $p>1$, $W^{s(\cdot,\cdot),p}(\Omega)$ is reflexive (see \cite[Proposition 3.25]{BREZIS2011}).

By incorporating the techniques from \cite{ANDREU2008201, ANDREU20091815, MAZON2016810}, 
we study the well-posedness of \eqref{eqn:main} using Nonlinear Semigroup Theory, 
with relevant theoretical results found in 
\cite{BENILAN199141} as well as the Appendix in \cite{ANDREU2010} and 
the references therein. We 
denote by $\mathbf{J}_0$ and $\mathbf{P}_0$ the following sets of functions:
\[\begin{array}{c}
    \mathbf{J}_0 \coloneqq \{j:\mathbb{R} \to [0,+\infty], 
    \textrm{ convex and lower semi-continuous with } j(0)=0\}, \\[1ex]
    \mathbf{P}_0 \coloneqq \{q \in C^{\infty}(\mathbb{R}): 0 \leq q' \leq 1, 
    \operatorname{supp}(q') \textrm{ is compact, and } 0 \notin \operatorname{supp}(q)\}.
\end{array}\]
In \cite{BENILAN199141} a relation for $u,v \in L^1(\Omega)$ is defined by $u \ll v$ if 
and only if
\[\int_{\Omega} j(u) dx \leq \int_{\Omega} j(v) dx\]
for all $j \in \mathbf{J}_0$. 
An operator $A \in L^1(\Omega) \times L^1(\Omega)$ is 
completely accretive if given $u_i, v_i$ such that $v_i=Au_i$, $i = 1, 2$ then 
\[\int_{\Omega} (v_1-v_2)q(u_1-u_2) \geq 0\]
for all $q \in \mathbf{P}_0$.
An operator $A$ in $X$ is $m$-completely accretive in $X$ if $A$ is completely 
accretive and $R(I+\lambda A) = X$ for some $\lambda>0$.

\begin{theorem} (See \cite{BENILAN199141, MAZON2016810}) \label{thm:mild-solution}
    If $A$ is an $m$-completely accretive operator, then for 
    every $f \in \overline{\mathrm{Dom}(A)}$, there is a unique mild 
    solution of the abstract Cauchy problem 
    \begin{equation}
        \left\{
        \begin{array}{ll}
            \displaystyle \vspace{0.25em} \frac{d u}{d t} + Au \ni 0, \\ 
            u(0) = f.
        \end{array}
        \right. 
    \end{equation}
    Moreover, if $A$ is the subdifferential of a convex and lower 
    semi-continuous function in $L^2(\Omega)$ then the mild solution of 
    the above problem is a strong solution.
\end{theorem}

In this paper, we assume that $k(x,y) \in L^{\infty}(\Omega \times \Omega)$ is 
symmetric and there exist constants 
$0< \delta_1 < \delta_2$, and $C_1 > 0$, $C_2 > 0$ such that 
\[C_1 \chi_{\{|x-y|<\delta_1\}} \leq k(x,y) 
\leq C_2 \chi_{\{|x-y|<\delta_2\}}\]
for all $(x,y) \in \Omega \times \Omega$. Similar to Lemma 2.2 in \cite{GAO20224837}, 
we shall prove a lemma concerning boundedness.
\begin{lemma} \label{lem:bounded}
    Let $\Omega$ be an open set with bounded boundary. 
    Let $1 \leq p < 2$. If there exists a constant $C_0>0$ such that 
    \[\int_{\Omega} |u(x)|^2 dx + \int_{\Omega} \int_{\Omega} 
    \frac{k(x,y)}{|x-y|^{N+s(x,y)p}} \left|u(y)-u(x)\right|^p dxdy \leq C_0, \]
    then we have $u \in L^2(\Omega) \cap W^{s(\cdot,\cdot),p}(\Omega)$ with
    \[\int_{\Omega} |u(x)|^2 dx + \int_{\Omega} \int_{\Omega} 
    \frac{\left|u(y)-u(x)\right|^p}{|x-y|^{N+s(x,y)p}} dxdy 
    \leq C(C_1,C_0,\delta_1,N,|\Omega|).\]
\end{lemma}

\begin{proof}
    We only need to show that $[u]_{W^{s(\cdot,\cdot),p}(\Omega)}$ is bounded. 
    Since $k(x,y) \geq 0$ and by Hölder's inequality, we have 
    \begin{align*} 
    \left[u\right]_{W^{s(\cdot,\cdot),p}(\Omega)}^p 
    &= \int_{\Omega} \int_{\Omega} \frac{|u(x)-u(y)|^p}{|x-y|^{N+s(x,y)p}} d x d y \\ 
    &= \iint_{\{|x-y| < \delta_1\}} \frac{|u(x)-u(y)|^p}{|x-y|^{N+s(x,y)p}} d x d y
     + \iint_{\{|x-y| \geq \delta_1\}} \frac{|u(x)-u(y)|^p}{|x-y|^{N+s(x,y)p}} d x d y \\  
    &\leq  \frac{1}{C_1} \iint_{\{|x-y| < \delta_1\}} \frac{k(x, y)}{|x-y|^{N+s(x,y)p}} |u(x)-u(y)|^p d x d y \\
    & \quad + \frac{1}{\min\big\{1,\delta_1^{N+2}\big\}} \iint_{\{|x-y| \geq \delta_1\}} |u(x)-u(y)|^p d x d y \\
    &\leq  \frac{C_0}{C_1} + \frac{2^{p-1}}{\min\big\{1,\delta_1^{N+2}\big\}} \int_{\Omega} \int_{\Omega} \left(|u(x)|^p + |u(y)|^p\right) d x d y \\
    &\leq  \frac{C_0}{C_1} + \frac{2^{p}}{\min\big\{1,\delta_1^{N+2}\big\}} \left(\int_{\Omega} |u|^2 dx\right)^{\frac{p}{2}} |\Omega|^{2-\frac{p}{2}} \\
    &\leq  \frac{C_0}{C_1} + \frac{4\big(\|u\|_{L^2(\Omega)}^2+1\big) \big(|\Omega|^2+1\big)}{\min\big\{1,\delta_1^{N+2}\big\}} \\
    &= C(C_1,C_0,\delta_1,N,|\Omega|).
    \end{align*}
    This concludes the proof.
\end{proof}

\begin{remark} \label{rmk:bounded}
    From this lemma, we know that when $1 \leq p < 2$, for a 
    sequence $\{u_k\} \subset W^{s(\cdot,\cdot),p}(\Omega)$, if $\{u_k\}$ is bounded 
    in $L^2(\Omega)$ and there exists a positive constant $C$ independent of $k=1,2, \cdots$ 
    such that
    \[\int_{\Omega} \int_{\Omega} 
    \frac{k(x,y)}{|x-y|^{N+s(x,y)p}} \left|u_k(y)-u_k(x)\right|^p dxdy \leq C, \]
    then $\{u_k\}$ is also bounded in $W^{s(\cdot,\cdot),p}(\Omega)$. 
\end{remark}

We will mainly consider the weak solution to the problem \eqref{eqn:main}, 
being $\Omega$ an open set of class $C^{0,1}$ with bounded boundary, $N \geq 2$. As in 
\cite{MAZON2016810}, the weak solution is defined as follows.

\begin{definition}
Given $f \in L^2(\Omega)$, we say that $u$ is a weak solution of problem \eqref{eqn:main}  
in $[0,T]$, if it satisfies the following conditions:
\begin{enumerate}[(i)]
    \item $u \in W^{1,1}\left(0,T;L^2(\Omega)\right)$;
    \item $u(0,\cdot)=f$; 
    \item There exists a function $\eta(t,\cdot,\cdot) \in L^{\infty}(\Omega \times \Omega)$, 
    $\eta(t,x,y)=-\eta(t,y,x)$ for almost all $(x,y) \in \Omega \times \Omega$, 
    $\|\eta(t,\cdot,\cdot)\|_{L^{\infty}(\Omega \times \Omega)} \leq 1$ and 
    \[k(x,y)\eta(t,x,y) \in k(x,y) \mathrm{sign} (u(t,y)-u(t,x))\]
    for a.e. $(x,y) \in \Omega \times \Omega$, such that for 
    any $\varphi \in W^{s(\cdot,\cdot),1}(\Omega) \cap L^2(\Omega)$, 
    the following integral equality holds:
    \begin{equation}
    \frac{1}{2} \int_{\Omega} \int_{\Omega} \frac{k(x,y)}{|x-y|^{N+s(x,y)}}
    \eta(t,x,y)(\varphi(y)-\varphi(x)) d x d y 
    =-\int_{\Omega} \frac{\partial u}{\partial t} \varphi(x) d x \label{eqn:main-weak}
    \end{equation}
    for almost all $t \in (0,T)$.
\end{enumerate}
\end{definition}

When investigating the existence of the weak solution for 
problem \eqref{eqn:main}, we focus on a simple yet important class 
of functions $s(x,y)$ that satisfy $s(x,y)=s(|x-y|)$. Such functions 
clearly satisfy the condition
\begin{equation}
    s(x-z, y-z) = s(x,y), \quad \forall x,y,z \in \Omega. \label{eqn:trin}
\end{equation}

Now we state the main result as follows.

\begin{theorem} \label{thm:main}
Let \eqref{eqn:trin} be satisfied. For every $f \in L^2(\Omega)$ there exists a unique 
weak solution of problem \eqref{eqn:main} in $[0,T]$ for any $T>0$. 
\end{theorem}

The proof of Theorem \ref{thm:main} is given in Section \ref{5}. To complete the proof, 
we employ a method similar to that in \cite{MAZON2016810}, which involves firstly 
using Nonlinear Semigroup Theory to establish the well-posedness of variable-order 
fractional $p$-Laplacian evolution equations when $1<p<2$, and then proving the case
$p=1$ through an approximation method. In order to overcome the difficulties encountered 
in the approximation process, we place particular emphasis on studying the density 
properties of the space $W^{s(\cdot,\cdot),p}(\Omega)$ in the next section.

\section{Basic results of the spaces \texorpdfstring{$W^{s(\cdot,\cdot),p}(\Omega)$}{Lg}} \label{3} 

In this section, we prove a series of basic results of variable-order fractional 
Sobolev spaces, which will be utilized hereafter. We refer to 
\cite{DINEZZA2012521} and the references 
therein for more basic results of the usual fractional Sobolev spaces.
To begin with, we prove the embedding theorem of the 
spaces $W^{s(\cdot,\cdot),p}(\Omega)$.

\begin{lemma} \label{lem:subspace-embedding}
    Let $\Omega$ be a bounded open set in $\mathbb{R}^N$. Let $s_1$ be a constant  
    with $s^- \leq s_1 \leq s(x,y)$ a.e. on $\Omega \times \Omega$. Then 
    the space $W^{s(\cdot,\cdot),p}(\Omega)$ is continuously 
    embedded in $W^{s_1,p}(\Omega)$. 
\end{lemma}

\begin{proof}
    Let $u \in W^{s(\cdot,\cdot),p}(\Omega)$ and 
    \[C = \sup_{(x,y) \in \Omega \times \Omega} |x-y|^{s(x,y)-s_1},\]
    then
    \[\begin{aligned}    
    \frac{1}{C^p} \int_{\Omega} \int_{\Omega} \frac{|u(x)-u(y)|^p}{|x-y|^{N+s_1p}} d x d y  
    &= \int_{\Omega} \int_{\Omega} \frac{|u(x)-u(y)|^p}{|x-y|^{N+s(x,y)p}} \left(\frac{|x-y|^{s(x,y)-s_1}}{C}\right)^p  d x d y \\ 
    &\leq \int_{\Omega} \int_{\Omega} \frac{|u(x)-u(y)|^p}{|x-y|^{N+s(x,y)p}} d x d y\\ 
    &= \left[u\right]_{W^{s(\cdot,\cdot),p}(\Omega)}^p, 
    \end{aligned}\]
    therefore
    \[\left[u\right]_{W^{s_1,p}(\Omega)} \leq C\left[u\right]_{W^{s(\cdot,\cdot),p}(\Omega)},\]
    and further there exists a positive constant $C=C(N, s, s_1, \Omega)$ such that
    \[\|u\|_{W^{s_1,p}(\Omega)} \leq C\|u\|_{W^{s(\cdot,\cdot),p}(\Omega)}.\]
    This concludes the proof.
\end{proof}

\begin{theorem} \label{thm:embedding}
    Let $\Omega$ be an open set in $\mathbb{R}^N$ of class $C^{0,1}$ with bounded boundary, 
    for all $(x,y) \in \overline{\Omega} \times \overline{\Omega}$, $s(x,y)p < N$ and $1 \leq q < p_{s(x,y)}^* = \frac{Np}{N-s(x,y)p}$. Then the 
    space $W^{s(\cdot,\cdot),p}(\Omega)$ is continuously embedded in $L^q(\Omega)$. Moreover, 
    the embedding is compact.
\end{theorem}

\begin{proof}
    Since $q < p_{s(x,y)}^* = \frac{Np}{N-s(x,y)p}$ for all $(x,y) \in \overline{\Omega} \times \overline{\Omega}$, 
    we know that $q < p_{s_1}^* = \frac{Np}{N-s_1p}$, where $s_1 = \underset{(x,y) \in \overline{\Omega} \times \overline{\Omega}}{\mathrm{min}}~s(x,y)$. 
    By Lemma \ref{lem:subspace-embedding} we know that $u \in W^{s_1,p}(\Omega)$, and there 
    exists a positive constant $C=C(N, s, s_1, \Omega)$ such that   
    \begin{equation} \label{eqn:subspace-em}
        \|u\|_{W^{s_1,p}(\Omega)} \leq C\|u\|_{W^{s(\cdot,\cdot),p}(\Omega)}.
    \end{equation}
    In view of the usual fractional Sobolev embedding theorem 
    (see \cite[Corollary 7.2]{DINEZZA2012521}), there 
    exists a positive constant $C=C(N, p, s_1, \Omega)$ such that
    \begin{equation} \label{eqn:fractional-em}
        \|u\|_{L^q(\Omega)} \leq C \|u\|_{W^{s_1,p}(\Omega)}.
    \end{equation} 
    Combining \eqref{eqn:subspace-em} and \eqref{eqn:fractional-em}, we know that 
    there exists a positive constant $C=C(N, p, s_1, \Omega)$ such that
    \[\|u\|_{L^q(\Omega)} \leq C \|u\|_{W^{s(\cdot,\cdot),p}(\Omega)}.\]
    
    Finally, we prove the compactness of this embedding. Let $\{u_k\}$ be 
    a bounded sequence in $W^{s(\cdot,\cdot),p}(\Omega)$, then $\{u_k\}$ is 
    also bounded in $W^{s_1,p}(\Omega)$. Since $q < p_{s_1}^*$, using the usual fractional Sobolev 
    compact embedding theorem, we know that $W^{s_1,p}(\Omega)$ is compactly embedded 
    in $L^q(\Omega)$. Therefore, $\{u_k\}$ has a convergent subsequence in $L^q(\Omega)$. 
    This fact concludes the proof. 
\end{proof}

Next, we prove the variable-order fractional Poincaré-Wirtinger inequality.

\begin{theorem} \label{thm:poincare}
    Let $\Omega$ be an open set in $\mathbb{R}^N$ of class $C^{0,1}$ with bounded boundary, 
    and let $u \in W^{s(\cdot,\cdot),p}(\Omega)$, then there exists a 
    constant $C=C(N, s, p, \Omega)$ such that 
    \[\int_{\Omega} \left|u(x)-u_{\Omega}\right| dx 
    \leq C [u]_{W^{s(\cdot,\cdot),p}(\Omega)},\]
    where $u_{\Omega} = \displaystyle \frac{1}{|\Omega|}\int_{\Omega} u dx$ is 
    the mean value of $u$ in $\Omega$.
\end{theorem}

\begin{proof}
    By Lemma \ref{lem:subspace-embedding} we know that $u \in W^{s^-,p}(\Omega)$, and there 
    exists a positive constant $C=C(N, s, \Omega)$ such that   
    \begin{equation} \label{eqn:subspace-em-2}
        [u]_{W^{s^-,p}(\Omega)} \leq C [u]_{W^{s(\cdot,\cdot),p}(\Omega)}.
    \end{equation}
    In view of the usual fractional Poincaré-Wirtinger 
    inequality (see \cite{HURRI201385}), there 
    exists a positive constant $C=C(N, p, s^-, \Omega)$ such that
    \begin{equation} \label{eqn:fractional-poincare}
        \int_{\Omega} \left|u(x)-u_{\Omega}\right| dx \leq C [u]_{W^{s^-,p}(\Omega)}.
    \end{equation} 
    Combining \eqref{eqn:subspace-em-2} and \eqref{eqn:fractional-poincare}, we know that 
    there exists a positive constant $C=C(N, p, s, \Omega)$ such that
    \[\int_{\Omega} \left|u(x)-u_{\Omega}\right| dx 
    \leq C [u]_{W^{s(\cdot,\cdot),p}(\Omega)},\]
    which concludes the proof.
\end{proof}

In the remaining part of this section, we establish a result concerning 
the density property of the spaces $W^{s(\cdot,\cdot),p}(\Omega)$. The proof of 
the density property is mainly based on a basic technique of convolution, joined 
with a cut-off, see \cite{MCLEAN2000, FISCELLA2015235}. We first present the 
properties related to the cut-off function.

\begin{lemma} \label{lem:cutoff}
    Let $\Omega \subset \mathbb{R}^N$ be an open set and let  
    $u \in W^{s(\cdot,\cdot),p}(\Omega)$. 
    Let $\tau \in C^{\infty}(\mathbb{R}^N)$ be compactly 
    supported, then $\tau u \in W^{s(\cdot,\cdot),p}(\Omega)$.  
    In addition, if $\operatorname{supp} \tau \subset \Omega$, 
    then $\tau u \in W^{s(\cdot,\cdot),p}(\mathbb{R}^N)$.
\end{lemma}

\begin{proof}
    It is clear that $\| \tau u \|_{L^p(\Omega)} < +\infty$ 
    since $\tau \in C_0^{\infty}(\mathbb{\mathbb{R}}^N)$. Adding and subtracting the factor 
    $\tau(x) u(y)$, we get 
    \[\begin{aligned}
     \left[\tau u\right]_{W^{s(\cdot,\cdot),p}(\Omega)} 
    & = \int_{\Omega} \int_{\Omega} \frac{|\tau(x) u(x)-\tau(y) u(y)|^p}{|x-y|^{N+s(x,y)p}} dxdy \\
    & \leq 2^{p-1} \left(\int_{\Omega} \int_{\Omega} \frac{|\tau(x) u(x)-\tau(x) u(y)|^p}{|x-y|^{N+s(x,y)p}} dxdy
      + \int_{\Omega} \int_{\Omega} \frac{|\tau(x) u(y)-\tau(y) u(y)|^p}{|x-y|^{N+s(x,y)p}} dxdy\right) \\
    & \leq 2^{p-1} \left(\int_{\Omega} \int_{\Omega} \frac{|\tau(x)|^p |u(x)- u(y)|^p}{|x-y|^{N+s(x,y)p}} dxdy
      + \int_{\Omega} \int_{\Omega} \frac{|u(x)|^p |\tau(x) -\tau(y)|^p}{|x-y|^{N+s(x,y)p}} dxdy\right).
    \end{aligned}\]  
    Since $u \in W^{s(\cdot,\cdot),p}(\Omega)$, we have 
    \begin{equation} \label{eqn:fmt}
        \int_{\Omega} \int_{\Omega} \frac{|\tau(x)|^p |u(x)- u(y)|^p}{|x-y|^{N+s(x,y)p}} dxdy 
        \leq \| \tau \|_{L^{\infty}(\Omega)}^p [u]_{W^{s(\cdot,\cdot),p}(\Omega)} < + \infty.
    \end{equation}

    On the other hand, since $\tau \in C^{1}(\Omega)$, we have 
    \begin{align} \label{eqn:lmt}
    & \int_{\Omega} \int_{\Omega} \frac{|u(x)|^p |\tau(x) -\tau(y)|^p}{|x-y|^{N+s(x,y)p}} dxdy \notag \\
    &\leq  \| \tau \|_{C^{1}(\Omega)}^p 
    \iint_{(\Omega \times \Omega) \cap \left\{(x,y):|x-y| \leq 1\right\}} \frac{|u(x)|^{p}}{|x-y|^{N+(s(x,y)-1)p}} d x d y \notag \\ 
    & \quad + \iint_{(\Omega \times \Omega) \cap \left\{(x,y):|x-y| > 1\right\}} \frac{|u(x)|^{p}}{|x-y|^{N+s(x,y)p}} d x d y \notag \\
    &\leq \| \tau \|_{C^{1}(\Omega)}^p 
    \int_{\Omega} \int_{\Omega \cap \left\{|y-x| \leq 1\right\}} \frac{|u(x)|^{p}}{|x-y|^{N+(s^--1)p}} d y d x \notag \\ 
    & \quad + \int_{\Omega}  \int_{\Omega \cap \left\{|y-x| > 1\right\}} \frac{|u(x)|^{p}}{|x-y|^{N+s^-p}} d y d x \notag \\
    &\leq  C \| u \|_{L^p(\Omega)}^p < + \infty,
    \end{align}
    where $C$ is a positive constant depends on $N$, $p$, $s^-$ and $\| \tau \|_{C^{1}(\Omega)}$. Note 
    that the last inequality follows from the fact that $|x-y|^{-N+(1-s^-)p}$ is integrable 
    with respect to $y$ if $|x-y| \leq 1$ since $N+(s^--1)p < N$ and $|x-y|^{-N-s^-p}$ is integrable  
    when $|x-y| > 1$ since $N+s^-p>N$. Combining \eqref{eqn:fmt} and \eqref{eqn:lmt}, we 
    know that $\tau u \in W^{s(\cdot,\cdot),p}(\Omega)$.
    
    In the case of $\operatorname{supp} \tau \subset \Omega$, 
    clearly $\tau u \in L^p(\mathbb{\mathbb{R}}^N)$. Hence, it remains to verify 
    that $[\tau u]_{W^{s(\cdot,\cdot),p}(\mathbb{R}^N)}$ is bounded. Using the symmetry of 
    the integral in $[\cdot]_{W^{s(\cdot,\cdot),p}(\mathbb{R}^N)}$ with respect to $x$ and 
    $y$, we can split as follows 
    \begin{equation} \label{eqn:split}
        [\tau u]_{W^{s(\cdot,\cdot),p}(\mathbb{R}^N)}^p 
    = [\tau u]_{W^{s(\cdot,\cdot),p}(\Omega)}^p 
    + 2 \int_{\Omega} \left( \int_{\mathbb{R}^N \backslash \Omega} \frac{|\tau(x) u(x)|^p}{|x-y|^{N+s(x,y)p}} d y \right) d x ,
    \end{equation}
    where the first term on the right side is finite 
    since $\tau u \in W^{s(\cdot,\cdot),p}(\Omega)$. Furthermore, set $K = \operatorname{supp} \tau u$, for 
    any $y \in \mathbb{R}^N \backslash K$,
    \[\begin{aligned} 
        \frac{|\tau(x) u(x)|^p}{|x-y|^{N+s(x,y)p}}  
    & = \chi_{K}(x)|\tau(x) u(x)|^p \frac{1}{|x-y|^{N+s(x,y)p}}  \\
    & \leq \chi_{K}(x) |\tau(x) u(x)|^p \sup_{x \in {K} } \frac{1}{|x-y|^{N+s(x,y)p}} \\ 
    & \leq \chi_{K}(x) |\tau(x) u(x)|^p \dfrac{1}{\min \left\{ \mathrm{dist}(y, \partial K)^{N+s^-p},\mathrm{dist}(y, \partial K)^{N+s^+p} \right\}}.  
    \end{aligned}\]
    Since $\mathrm{dist}(\partial \Omega, \partial K) \geq \alpha > 0$ and $N+s^+p \geq N+s^-p > N$, we get
    \begin{equation} \label{eqn:upperbound}
        \int_{\Omega} \left( \int_{\mathbb{R}^N \backslash \Omega} \frac{|\tau(x) u(x)|^p}{|x-y|^{N+s(x,y)p}} d y \right) d x < + \infty.
    \end{equation}
    Combining \eqref{eqn:split} and \eqref{eqn:upperbound}, we know that $\tau u \in W^{s(\cdot,\cdot),p}(\mathbb{R}^N)$.
\end{proof}

For any $\delta > 0$ and function $u$ we set $u_{\delta}(x) = u(x', x_{N}-\delta)$. The 
following lemma presents a property pertaining to the translations. 

\begin{lemma} \label{lem:tra}
    Let \eqref{eqn:trin} be satisfied and $u \in W^{s(\cdot,\cdot),p}(\mathbb{R}^N)$, 
    then $\|u-u_{\delta}\|_{W^{s(\cdot,\cdot),p}(\mathbb{R}^N)} \to 0$ as $\delta \to 0$.
\end{lemma}

\begin{proof} 
    Clearly $u \in L^{p}(\mathbb{R}^N)$. The continuity of the translations 
    in $L^p(\mathbb{R}^N)$ gives that 
    \[\int_{\mathbb{R}^N} |u(x) - u(x-\delta e_N)|^p dx \to 0\]
    as $\delta \to 0$, that is 
    \begin{equation} \label{eqn:ctlp}
    \|u-u_{\delta}\|_{L^{p}(\mathbb{R}^N)} \to 0 \quad \textrm{as} \quad \delta \to 0.
    \end{equation}

    Since $u \in W^{s(\cdot,\cdot),p}(\mathbb{R}^N)$, it is clear 
    that $v(x,y) = \frac{u(x)-u(y)}{|x-y|^{\frac{N}{p}+s(x,y)}}$ is 
    in $L^p(\mathbb{R}^N \times \mathbb{R}^N)$. Let $w = (e_N,e_N) \in \mathbb{R}^N \times \mathbb{R}^N$, 
    $\tilde{x} = (x, y) \in \mathbb{R}^N \times \mathbb{R}^N$, using \eqref{eqn:trin} and 
    the continuity of the translations in $L^p(\mathbb{R}^N \times \mathbb{R}^N)$ we get
    \[\int_{\mathbb{R}^N \times \mathbb{R}^N} |v(\tilde{x}) - v(\tilde{x}-\delta w)|^p d \tilde{x} \to 0\]
    as $\delta \to 0$, that is 
    \begin{equation} \label{eqn:ctvfs}
    [u-u_{\delta}]_{W^{s(\cdot,\cdot),p}(\mathbb{R}^N)} \to 0 \quad \textrm{as} \quad \delta \to 0,
    \end{equation}
    then \eqref{eqn:ctlp} and \eqref{eqn:ctvfs} conclude the proof.
\end{proof}

Denote by $B_R$ the ball centered at $0$ with radius $R$. Let $\eta \in C_0^{\infty}(\mathbb{R}^N)$ 
be such that $\eta \geq 0$, $\operatorname{supp} \eta \subset B_1$ 
and $\int_{B_1} \eta(x) dx = 1$. Let $\varepsilon > 0$, for $x \in \mathbb{R}^N$,
the mollifier $\eta_{\varepsilon}$ is defined as 
$\eta_{\varepsilon}(x) = \frac{1}{\varepsilon^n} \eta\left(\frac{x}{\varepsilon}\right)$. 
We set $u_{\varepsilon}(x) = (u * \varepsilon)(x)$. The following lemma 
provides an approximation property on the whole space $\mathbb{R}^N$.

\begin{lemma} \label{lem:mollify}
    Let \eqref{eqn:trin} be satisfied and $u \in W^{s(\cdot,\cdot),p}(\mathbb{R}^N)$, 
    then $\|u-u_{\varepsilon}\|_{W^{s(\cdot,\cdot),p}(\mathbb{R}^N)} \to 0$ as $\varepsilon \to 0^+$. 
\end{lemma}

\begin{proof}
    It is sufficient to show that $[u-u_{\varepsilon}]_{W^{s(\cdot,\cdot),p}(\mathbb{R}^N)} \to 0$ 
    as $\varepsilon \to 0^+$. Using the Hölder's inequality in combination with 
    Tonelli's and Fubini's theorems, we obtain 
    \begin{align} \label{eqn:fubini}
        &[u-u_{\varepsilon}]_{W^{s(\cdot,\cdot),p}(\mathbb{R}^N)} \notag \\ 
        &=\int_{\mathbb{R}^{N} \times \mathbb{R}^{N}} \frac{\left|u(x)-u_{\varepsilon}(x)-u(y)+u_{\varepsilon}(y)\right|^{p}}{|x-y|^{N+s(x,y)p}} d x d y \notag \\
        &=\int_{\mathbb{R}^{N} \times \mathbb{R}^{N}} \frac{\left|u(x)-u(y)-\int_{\mathbb{R}^{N}}(u(x-z)-u(y-z)) \eta_{\varepsilon}(z) d z\right|^{p}}{|x-y|^{N+s(x,y)p}} 
         d x d y \notag \\
        &=\int_{\mathbb{R}^{N} \times \mathbb{R}^{N}} \frac{\left|u(x)-u(y)-\frac{1}{\varepsilon^n} \int_{B_{\varepsilon}}(u(x-z)-u(y-z)) \eta\left(\frac{z}{\varepsilon}\right) d z\right|^{p}}{|x-y|^{N+s(x,y)p}} 
         d x d y \notag \\
        &=\int_{\mathbb{R}^{N} \times \mathbb{R}^{N}} \frac{\left|\int_{B_{1}}(u(x)-u(y)-u(x-\varepsilon \tilde{z})+u(y-\varepsilon \tilde{z})) \eta(\tilde{z}) d \tilde{z}\right|^{p}}{|x-y|^{N+s(x,y)p}} 
         d x d y \notag \\
        &\leq\left|B_{1}\right|^{p-1} \int_{\mathbb{R}^{N} \times \mathbb{R}^{N}}\left(\int_{B_{1}}\frac{|u(x)-u(y)-u(x-\varepsilon z)+u(y-\varepsilon z)|^{p}}{|x-y|^{N+s(x,y)p}} \eta^{p}(z) d z\right) d x d y \notag \\
        &=\left|B_{1}\right|^{p-1} \int_{\left(\mathbb{R}^{N} \times \mathbb{R}^{N}\right) \times B_{1}}\frac{|u(x)-u(y)-u(x-\varepsilon z)+u(y-\varepsilon z)|^{p}}{|x-y|^{N+s(x,y)p}} \eta^{p}(z) d x d y d z .
    \end{align}
    
    Recall that the function $v(x,y)$ we define in Lemma \ref{lem:tra} is 
    in $L^p(\mathbb{R}^N \times \mathbb{R}^N)$. Fix $z \in B_1$, 
    let $w=(z,z)$ and $\tilde{x} = (x, y)$, using \eqref{eqn:trin} and 
    the continuity of the translations in $L^p(\mathbb{R}^N \times \mathbb{R}^N)$ we get
    \[\int_{\mathbb{R}^N \times \mathbb{R}^N} |v(\tilde{x}) - v(\tilde{x}-\varepsilon w)|^p d \tilde{x} \to 0\]
    as $\varepsilon \to 0^+$, that is 
    \[\int_{\mathbb{R}^N \times \mathbb{R}^N} \frac{|u(x)-u(y)-u(x-\varepsilon z)+u(y-\varepsilon z)|^{p}}{|x-y|^{N+s(x,y)p}} d x d y \to 0\]
    as $\varepsilon \to 0^+$.

    Moreover, for a.e. $z \in B_1$, by \eqref{eqn:trin} we have 
    \begin{align} \label{eqn:dct}
    & \eta^{p}(z) \int_{\mathbb{R}^N \times \mathbb{R}^N} \frac{|u(x)-u(y)-u(x-\varepsilon z)+u(y-\varepsilon z)|^{p}}{|x-y|^{N+s(x,y)p}} d x d y \notag \\
    & \leq  2^{p-1} \eta^{p}(z) \left(\int_{\mathbb{R}^N \times \mathbb{R}^N} \frac{|u(x)-u(y)|^{p}}{|x-y|^{N+s(x,y)p}} d x d y 
     + \int_{\mathbb{R}^N \times \mathbb{R}^N} \frac{|u(x-\varepsilon z)-u(y-\varepsilon z)|^{p}}{|x-y|^{N+s(x,y)p}} d x d y \right) \notag \\
    &= 2^{p} \eta^{p}(z) \int_{\mathbb{R}^N \times \mathbb{R}^N} \frac{|u(x)-u(y)|^{p}}{|x-y|^{N+s(x,y)p}} d x d y 
    \in L^{\infty}(B_1)
    \end{align} 
    for any $\varepsilon > 0$. Hence, by \eqref{eqn:fubini}, \eqref{eqn:dct} and dominated 
    convergence theorem we have 
    \[\left|B_{1}\right|^{p-1} \int_{\left(\mathbb{R}^{N} \times \mathbb{R}^{N}\right) \times B_{1}}\frac{|u(x)-u(y)-u(x-\varepsilon z)+u(y-\varepsilon z)|^{p}}{|x-y|^{N+s(x,y)p}} \eta^{p}(z) d x d y d z \to 0\]
    as $\varepsilon \to 0^+$. This fact concludes the proof.
\end{proof}

Next, we prove a density result on a bounded Lipschitz hypograph. For the definition
of a Lipschitz hypograph, see \cite[Definition 3.28]{MCLEAN2000}.

\begin{lemma} \label{lem:hyp}
    Let \eqref{eqn:trin} be satisfied, and let $\Omega \subset \mathbb{R}^N$ be 
    a Lipschitz hypograph with bounded boundary. Given $\varepsilon > 0$, 
    let $u \in W^{s(\cdot,\cdot),p}(\Omega)$, 
    then there exists $v \in C^{\infty}(\mathbb{R}^N)$ such 
    that $\|v-u\|_{W^{s(\cdot,\cdot),p}(\Omega)} < \varepsilon$.
\end{lemma}

\begin{proof}
    Since $\Omega \subset \mathbb{R}^N$ is a Lipschitz hypograph, 
    there exists a Lipschitz function $\gamma: \mathbb{R}^{N-1} \to \mathbb{R}^N$ having 
    compact support such that 
    \[\Omega = \{(x',x_N) \in \mathbb{R}^{N-1} \times \mathbb{R}: x_N < \gamma(x')\},\]
    for $\delta>0$, we define 
    \[\Omega_{\delta} = \{x \in \mathbb{R}^N: x_N < \gamma(x') + \delta\},\]
    then $\Omega \subset \Omega_{\delta}$. By \eqref{eqn:trin} we 
    have $u_{\delta} \in W^{s(\cdot,\cdot),p}(\Omega_{\delta})$. Let $\varepsilon > 0$, by 
    Lemma \ref{lem:tra}, we can choose $\delta$ small enough so that 
    \begin{equation} \label{eqn:tra-ori}
    \|u_{\delta}-u\|_{W^{s(\cdot,\cdot),p}(\Omega)} < \frac{\varepsilon}{2}. 
    \end{equation}
    Now choose a cut-off function $\tau \in C^{\infty}(\mathbb{R}^N)$ satisfying 
    \[\tau(x)=\left\{
        \begin{array}{ll}
            1, & \textrm{if~~} x \in \Omega,  \\[1ex]
            0, & \textrm{if~~} x \in  \mathbb{R}^N \backslash \Omega_{\frac{\delta}{2}}.
        \end{array}
    \right.\]
    By Lemma \ref{lem:cutoff}, we know 
    that $\tau u_{\delta} \in W^{s(\cdot,\cdot),p}(\mathbb{R}^N)$. Hence, by 
    Lemma \ref{lem:mollify}, there exists $v \in C^{\infty}(\mathbb{R}^N)$ such that 
    \begin{equation} \label{eqn:mol-cuttra}
        \|v - \tau u_{\delta}\|_{W^{s(\cdot,\cdot),p}(\mathbb{R}^N)} < \frac{\varepsilon}{2}.
    \end{equation}
    Finally, combining \eqref{eqn:tra-ori} 
    with \eqref{eqn:mol-cuttra} we obtain
    \[\begin{aligned}
        \left\|v - u\right\|_{W^{s(\cdot,\cdot),p}(\Omega)} 
    & \leq \left\|v - \tau u_{\delta}\right\|_{W^{s(\cdot,\cdot),p}(\Omega)} 
     + \left\|\tau u_{\delta} - u\right\|_{W^{s(\cdot,\cdot),p}(\Omega)} \\
    & \leq \|v - \tau u_{\delta}\|_{W^{s(\cdot,\cdot),p}(\mathbb{R}^N)} 
     + \|u_{\delta}-u\|_{W^{s(\cdot,\cdot),p}(\Omega)} \\
    & < \frac{\varepsilon}{2} + \frac{\varepsilon}{2} = \varepsilon.
    \end{aligned}\]
    This concludes the proof.
\end{proof}

Finally, combining Lemma \ref{lem:cutoff} - Lemma \ref{lem:hyp}, 
the global approximation by smooth functions on a bounded Lipschitz domain is 
proven by a partition of unity technique.

\begin{theorem} \label{thm:density}
    Let \eqref{eqn:trin} be satisfied, and let $\Omega \subset \mathbb{R}^N$ be 
    an open set of class $C^{0,1}$ with bounded boundary. Then $C^{\infty}(\overline{\Omega})$ is 
    dense in $W^{s(\cdot,\cdot),p}(\Omega)$.
\end{theorem}

\begin{proof}
    Since $\Omega$ is a Lipschitz domain with bounded boundary, there exist finite 
    families $\{U_{j}\}_{j=1}^k$ and $\{\Omega_{j}\}_{j=1}^k$ have the following properties:
    \begin{enumerate}[(i)]
        \item $\displaystyle \partial \Omega \subset \bigcup_{j=1}^k U_{j}$;
        \item Each $\Omega_{j}$ can be transformed to a Lipschitz hypograph by a rotation plus a translation; 
        \item For each $j \in \{1,\cdots,k\}$, $U_j \cap \Omega = U_j \cap \Omega_j$.
    \end{enumerate}

    Define one additional open set $U_0 = \{x \in \Omega: \mathrm{dist}(x, \mathbb{R}^N 
    \backslash \Omega) > \delta\}$, choosing $\delta>0$ small enough so that 
    \[\overline{\Omega} \subset \bigcup_{j=0}^k U_{j}.\]
    If we consider the covering, there exists a partition of unity related to it, i.e. there 
    exists a family $\{\tau_{j}\}_{j=0}^k$ such that $\tau_{j} \in C_0^{\infty}(U_j)$, 
    $0 \leq \tau_{j} \leq 1$ for all $j \in \{0,\cdots,k\}$ and $\sum_{j=0}^k \tau_{j} = 1$. For 
    any $j \in \{1,\cdots,k\}$, by Lemma \ref{lem:cutoff}, we know 
    that $\tau_j u \in W^{s(\cdot,\cdot),p}(U_j \cap \Omega_j)$. Fix $\varepsilon > 0$.
    By Lemma \ref{lem:hyp}, if $1 \leq j \leq k$ then 
    there exists $v_j \in C^{\infty}(\mathbb{R}^N)$ such that 
    \[\left\|v_j - \tau_j u\right\|_{W^{s(\cdot,\cdot),p}(U_j \cap \Omega_j)} < \frac{\varepsilon}{2k}.\]
    Also, by Lemma \ref{lem:cutoff}, we know that $\tau_0 u \in W^{s(\cdot,\cdot),p}(\mathbb{R}^N)$. By 
    Lemma \ref{lem:mollify}, there exists $v_0 \in C^{\infty}(\mathbb{R}^N)$ such that
    \[\left\|v_0 - \tau_0 u\right\|_{W^{s(\cdot,\cdot),p}(\mathbb{R}^N)} < \frac{\varepsilon}{2},\]

    Now, define $v = \sum_{j=0}^k v_j$. Then clearly $v \in C^{\infty}(\overline{\Omega})$. 
    Since $u = \sum_{j=0}^k \tau_{j} u$, we obtain that 
    \[\begin{aligned}
        \left\|v - u\right\|_{W^{s(\cdot,\cdot),p}(\Omega)}
    &= \left\|\sum_{j=0}^k (v  - \tau_j u)\right\|_{W^{s(\cdot,\cdot),p}(\Omega)} \\
    & \leq \left\|v_0 - \tau_0 u\right\|_{W^{s(\cdot,\cdot),p}(U_0)}
    + \sum_{j=1}^k \left\|v_j - \tau_j u\right\|_{W^{s(\cdot,\cdot),p}(U_j \cap \Omega_j)} \\
    & < \frac{\varepsilon}{2} + \frac{\varepsilon}{2}
    = \varepsilon,
    \end{aligned}\]
    The arbitrariness of $\varepsilon$ concludes the proof.
\end{proof}

\section{The variable-order fractional \texorpdfstring{$p$}{Lg}-Laplacian evolution equations} \label{4}
In this section, we will study an approximating problem for \eqref{eqn:main} and 
prove the existence and uniqueness of its weak solution. Let $\Omega \subset \mathbb{R}^N$ be 
an open set of class $C^{0,1}$ with bounded boundary. For $1<p<2$, we consider the 
variable-order fractional $p$-Laplacian evolution equations
\begin{equation}
    \left\{
    \begin{array}{ll}
        \displaystyle \vspace{0.2em} \frac{\partial u_p}{\partial t} = \int_{\Omega} 
        \frac{k(x,y)}{|x-y|^{N+s(x,y)p}} U_p(t,x,y)
         d y, & \textrm{in~~} (0,T) \times \Omega,  \\ 
        u(0,x) = f, & \textrm{on~~} \Omega.
    \end{array}
    \right. \label{eqn:vfpl}
\end{equation}
where $U_p(t,x,y) = \left|u_p(t,y)-u_p(t,x)\right|^{p-2}\left(u_p(t,y)-u_p(t,x)\right)$ and 
the corresponding definition is described as follows:

\begin{definition}
    Given $f \in L^2(\Omega)$, we say that $u_p$ is a weak solution of problem \eqref{eqn:vfpl}  
    in $[0,T]$, if $u_p \in W^{1,1}\left(0,T;L^2(\Omega)\right)$, $u_p(0,\cdot)=f$ and for 
    any $\varphi \in W^{s(\cdot,\cdot),p}(\Omega) \cap L^2(\Omega)$, 
    the following integral equality holds:
    \begin{equation}
        \frac{1}{2} \int_{\Omega} \int_{\Omega} \frac{k(x,y)}{|x-y|^{N+s(x,y)p}}
        U_p(t,x,y)
        (\varphi(y)-\varphi(x)) d x d y 
        =-\int_{\Omega} \frac{\partial u_p}{\partial t} \varphi(x) d x
    \end{equation}
    for almost all $t \in (0,T)$.
\end{definition}

To study the problem \eqref{eqn:vfpl} we consider the energy 
functional $\mathcal{B}_p^{s(\cdot,\cdot)}: L^2(\Omega) \to [0,+\infty]$ given by 
\[\mathcal{B}_p^{s(\cdot,\cdot)}(u_p)=\left\{
    \begin{array}{ll}
        \displaystyle \vspace{0.2em} \frac{1}{2} \int_{\Omega} \int_{\Omega} 
         \frac{k(x,y)\left|u_p(y)-u_p(x)\right|^{p}}{|x-y|^{N+s(x,y)p}} 
        d x d y, & \textrm{if~~} u_p \in L^2(\Omega) \cap W^{s(\cdot,\cdot),p}(\Omega),  \\ 
        + \infty, & \textrm{if~~} u_p \in L^2(\Omega) \backslash W^{s(\cdot,\cdot),p}(\Omega).
    \end{array}
\right.\]
By Fatou's Lemma we have that $\mathcal{B}_p^{s(\cdot,\cdot)}$ is lower semi-continuous  
in $L^2(\Omega)$. Then, since $\mathcal{B}_p^{s(\cdot,\cdot)}$ is convex, we know 
that $\mathcal{B}_p^{s(\cdot,\cdot)}$ is weak lower semi-continuous in $L^2(\Omega)$ and 
the subdifferential $\partial \mathcal{B}_p^{s(\cdot,\cdot)}$ is a maximal monotone operator 
in $L^2(\Omega)$. To characterize $\partial \mathcal{B}_p^{s(\cdot,\cdot)}$ we introduce the 
following operator:

\begin{definition}
    We define the operator $B_p^{s(\cdot,\cdot)}$ in $L^2(\Omega) \times 
    L^2(\Omega)$ by $v_p \in B_p^{s(\cdot,\cdot)}u_p$ if and only 
    if $u_p \in W^{s(\cdot,\cdot),p}(\Omega) \cap L^2(\Omega)$, $v_p \in L^2(\Omega)$ and 
    \begin{equation}  \label{eqn:elliptic-vfpl}
        \frac{1}{2} \int_{\Omega} \int_{\Omega} \frac{k(x,y)}{|x-y|^{N+s(x,y)p}} U_p(x,y)
    (\varphi(y)-\varphi(x)) d x d y =\int_{\Omega} v_p(x) \varphi(x) d x
    \end{equation}
    for all $\varphi \in W^{s(\cdot,\cdot),p}(\Omega) \cap L^2(\Omega)$, 
    where $U_p(x,y) = \left|u_p(y)-u_p(x)\right|^{p-2}\left(u_p(y)-u_p(x)\right)$.
\end{definition}

In the following result, we prove that the operator $B_p^{s(\cdot,\cdot)}$ satisfies adequate 
conditions for the application of the Nonlinear Semigroup Theory. 

\begin{theorem} \label{thm:vfpl}
    The operator $B_p^{s(\cdot,\cdot)}$ is $m$-completely 
    accretive in $L^2(\Omega)$ with dense domain. Moreover, $B_p^{s(\cdot,\cdot)} 
    = \partial \mathcal{B}_p^{s(\cdot,\cdot)}$.
\end{theorem}

\begin{proof}
    Let $v_{p}^{(i)} \in B_p^{s(\cdot,\cdot)}u_{p}^{(i)}$, $i=1,2$ and $q \in \mathbf{P}_0$. Since 
    $u_{p}^{(1)}, u_{p}^{(2)} \in W^{s(\cdot,\cdot),p}(\Omega)$, we 
    have $q\big(u_{p}^{(1)}-u_{p}^{(2)}\big) \in W^{s(\cdot,\cdot),p}(\Omega) \cap L^{\infty}(\Omega)$. 
    Then we can take $q\big(u_{p}^{(1)}-u_{p}^{(2)}\big)$ as test function 
    in \eqref{eqn:elliptic-vfpl} and we get 
    \begin{align*} 
    &\int_{\Omega} \big(v_{p}^{(1)}-v_{p}^{(2)}\big)q\big(u_{p}^{(1)}-u_{p}^{(2)}\big)dx \\
       &= \frac{1}{2} \int_{\Omega} \int_{\Omega} \frac{k(x,y)U_{p}^{(1)}(x,y)}{|x-y|^{N+s(x,y)p}} 
       \big(q\big(u_{p}^{(1)}(y)-u_{p}^{(2)}(y)\big) - q\big(u_{p}^{(1)}(x)-u_{p}^{(2)}(x)\big)\big) d x d y \\
       &\quad - \frac{1}{2} \int_{\Omega} \int_{\Omega} \frac{k(x,y)U_{p}^{(2)}(x,y)}{|x-y|^{N+s(x,y)p}} 
       \big(q\big(u_{p}^{(1)}(y)-u_{p}^{(2)}(y)\big) - q\big(u_{p}^{(1)}(x)-u_{p}^{(2)}(x)\big)\big) d x d y \\
       &= \frac{1}{2} \int_{\Omega} \int_{\Omega} \frac{k(x,y)\big(U_{p}^{(1)}(x,y)-U_{p}^{(2)}(x,y)\big)}{|x-y|^{N+s(x,y)p}} 
       \big(q\big(u_{p}^{(1)}(y)-u_{p}^{(2)}(y)\big) - q\big(u_{p}^{(1)}(x)-u_{p}^{(2)}(x)\big)\big) d x d y.
    \end{align*}
    The right side of the above formula is nonnegative since $t \mapsto t|t|^{p-2}$ is monotonically 
    increasing. Hence, 
    \[\int_{\Omega} \big(v_{p}^{(1)}-v_{p}^{(2)}\big)q\big(u_{p}^{(1)}-u_{p}^{(2)}\big)dx \geq 0,\]
    from which it follows that $B_p^{s(\cdot,\cdot)}$ is completely accretive.

    Let us show that the operator $B_p^{s(\cdot,\cdot)}$ satisfies the range 
    condition $L^2(\Omega) \subset R \big( I + B_p^{s(\cdot,\cdot)} \big)$. 
    Given $f \in L^2(\Omega)$, we consider the functional 
    \[J(u) = \mathcal{B}_p^{s(\cdot,\cdot)}(u) + \frac{1}{2} \int_{\Omega} u^2 dx - \int_{\Omega} fu dx. \]
    Now we show that $J(u)$ admits a unique minimizer in $L^2(\Omega)$. By Cauchy's inequality 
    with $\varepsilon$, we have, for any $u \in L^{2}(\Omega)$,
    \[J(u) 
    \geq \frac{1}{2} \int_{\Omega} u^{2} dx-\int_{\Omega} f u dx 
    \geq \left(\dfrac{1}{2}-\frac{\varepsilon}{2}\right)\int_{\Omega} u^{2} dx 
    -\dfrac{1}{2\varepsilon} \int_{\Omega} f^{2} dx,\] 
    as long as we take $0 < \varepsilon < 1$, we obtain that $J(u)$ is bounded from below and 
    hence $\underset{u \in L^{2}(\Omega)}{\mathrm{inf}}~ J(u)$ is a finite number. The 
    definition of infimum then implies that there 
    exists a minimizing sequence $u_k \in W^{s(\cdot,\cdot), p}(\Omega) \cap L^{2}(\Omega)$. 
    Existence of the limit $\displaystyle \lim_{k \to \infty} J(u_k)$ implies the boundedness of $J(u_k)$, i.e. 
    for some constant $M$, 
    \[\mathcal{B}_p^{s(\cdot,\cdot)}(u_k) + \frac{1}{2} \int_{\Omega} u_k^2 dx - \int_{\Omega} f u_k dx \leq M, \quad k = 1,2, \cdots.\]
    Using Cauchy's inequality with $\varepsilon$ again, we obtain 
    \begin{equation} \label{eqn:functionalbound}
        \mathcal{B}_p^{s(\cdot,\cdot)}(u_k)+\left(\frac{1}{2} - \frac{\varepsilon}{2}\right) 
        \int_{\Omega} u_k^{2} dx \leq M + \frac{1}{2\varepsilon}\int_{\Omega} f^2 dx,
    \end{equation}
    By Lemma \ref{rmk:bounded}, we know that $\{u_k\}$ is bounded in $W^{s(\cdot,\cdot), p}(\Omega)$. Since 
    $1<p<2$, by the reflexivity of $W^{s(\cdot,\cdot), p}(\Omega)$, we can assume, taking a subsequence if necessary, 
    that $u_k \rightharpoonup u_p$ weakly in $W^{s(\cdot,\cdot), p}(\Omega)$. Moreover, 
    by \eqref{eqn:functionalbound}, we have $\{u_k\}$ is bounded in $L^2(\Omega)$, and 
    consequently $u_p \in L^2(\Omega)$. Thus,
    \[\inf _{u \in L^{2}(\Omega)} J(u) \leq J(u_p) 
    \leq \underset{k \to \infty}{\underline{\text{lim}}} J\left(u_{k}\right) 
    \leq \lim _{k \rightarrow \infty} J\left(u_{k}\right)=\inf _{u \in L^{2}(\Omega)} J(u),\]
    from which we deduce that $u_p$ is a minimizer of the functional $J(u)$. The uniqueness 
    follows by the strict convexity of $J(u)$. Now, we derive the Euler-Lagrange equation satisfied 
    by $u_p$. Fix a function $\varphi \in W^{s(\cdot,\cdot), p}(\Omega) \cap L^{2}(\Omega)$, then 
    the function
    \[F(t) \coloneqq \mathcal{B}_p^{s(\cdot,\cdot)}(u_p+t\varphi) + \frac{1}{2} \int_{\Omega} (u_p+t\varphi)^2 dx 
    - \int_{\Omega} f(u_p+t\varphi) dx\]
    has a minimum at $t=0$, and consequently
    \[\frac{1}{2} \int_{\Omega} \int_{\Omega} \frac{k(x,y)U_p(x,y)}{|x-y|^{N+s(x,y)p}} 
    (\varphi(y)-\varphi(x)) d x d y -\int_{\Omega} (f(x)-u_p(x)) \varphi(x) d x
    = F'(0) = 0.\]
    Then, we have $f-u_p \in B_p^{s(\cdot,\cdot)}u_p$ and thus $B_p^{s(\cdot,\cdot)}$ is $m$-completely 
    accretive in $L^2(\Omega)$.

    Let us now show that $\mathrm{Dom}\big(B_p^{s(\cdot,\cdot)}\big)$ is dense in $L^2(\Omega)$. To 
    see the fact it is enough to show that 
    \[W^{s(\cdot,\cdot), p}(\Omega) \cap L^{2}(\Omega) \subset 
    \overline{\mathrm{Dom}\big(B_p^{s(\cdot,\cdot)}\big)} ^ {L^{2}(\Omega)}.\] 
    So, let us take $v_p \in W^{s(\cdot,\cdot), p}(\Omega) \cap L^{2}(\Omega)$. Since $B_p^{s(\cdot,\cdot)}$ is $m$-completely 
    accretive in $L^2(\Omega)$, there exists $u_{p}^{(n)} \in \mathrm{Dom}\big(B_p^{s(\cdot,\cdot)}\big)$ 
    such that $n(v_p-u_{p}^{(n)}) \in B_p^{s(\cdot,\cdot)} u_{p}^{(n)}$, i.e.
    \[\frac{1}{2} \int_{\Omega} \int_{\Omega} \frac{k(x,y)}{|x-y|^{N+s(x,y)p}} U_{p}^{(n)}(x,y)
    (\varphi(y)-\varphi(x)) d x d y =n \int_{\Omega} (v_p(x)-u_{p}^{(n)}(x)) \varphi(x) d x\]
    for all $\varphi \in W^{s(\cdot,\cdot), p}(\Omega) \cap L^{2}(\Omega)$. Then, 
    taking $\varphi = v_p-u_{p}^{(n)}$, and applying Young's inequality, we obtain that 
    \begin{align*} 
        \int_{\Omega} (v_p(x)-u_{p}^{(n)}(x))^2 d x 
       =\,& \frac{1}{2n} \int_{\Omega} \int_{\Omega} \frac{k(x,y)}{|x-y|^{N+s(x,y)p}} U_{p}^{(n)}(x,y)
       (v_p(y)-v_p(x)) d x d y \\
       &- \frac{1}{2n} \int_{\Omega} \int_{\Omega} \frac{k(x,y)}{|x-y|^{N+s(x,y)p}} 
       |u_{p}^{(n)}(y)-u_{p}^{(n)}(x)|^p d x d y \\
       \leq \, & \frac{1}{2np} \int_{\Omega} \int_{\Omega} \frac{k(x,y)}{|x-y|^{N+s(x,y)p}} 
       |v_p(y)-v_p(x)|^p d x d y \\
       &-\frac{1}{2np} \int_{\Omega} \int_{\Omega} \frac{k(x,y)}{|x-y|^{N+s(x,y)p}} 
       |u_{p}^{(n)}(y)-u_{p}^{(n)}(x)|^p d x d y\\
       \leq \, & \frac{C_2}{2np} [v_p]_{W^{s(\cdot,\cdot), p}(\Omega)}^p,
    \end{align*}
    from which it follows that $u_{p}^{(n)} \to v_p$ in $L^2(\Omega)$.

    Finally, let us show that $B_p^{s(\cdot,\cdot)} 
    = \partial \mathcal{B}_p^{s(\cdot,\cdot)}$. Let $v_p \in B_p^{s(\cdot,\cdot)} u_p$, we have 
    \[\frac{1}{2} \int_{\Omega} \int_{\Omega} \frac{k(x,y)}{|x-y|^{N+s(x,y)p}} U_p(x,y)
    (\varphi(y)-\varphi(x)) d x d y =\int_{\Omega} v_p(x) \varphi(x) d x\]
    for all $\varphi \in W^{s(\cdot,\cdot), p}(\Omega) \cap L^{2}(\Omega)$. Then, given 
    $w_p \in W^{s(\cdot,\cdot), p}(\Omega) \cap L^{2}(\Omega)$, taking $\varphi = w_p-u_p$ and 
    using the numerical inequality $p|a|^{p-2}a(b-a) \leq |b|^p - |a|^p$ we obtain
    \begin{align*} 
       & \int_{\Omega} v_p(x) (w_p(x)-u_p(x)) d x \\
       &= \frac{1}{2} \int_{\Omega} \int_{\Omega} \frac{k(x,y)}{|x-y|^{N+s(x,y)p}} 
       U_p(x,y)(w_p(y)-w_p(x)-(u_p(y)-u_p(x))) d x d y \\
       &\leq  \frac{1}{2} \int_{\Omega} \int_{\Omega} \frac{k(x,y)}{|x-y|^{N+s(x,y)p}}
       |w_p(y)-w_p(x)|^p d x d y 
       - \frac{1}{2} \int_{\Omega} \int_{\Omega} \frac{k(x,y)}{|x-y|^{N+s(x,y)p}}
       |u_p(y)-u_p(x)|^p d x d y \\
       &=\mathcal{B}_p^{s(\cdot,\cdot)}(w_p) - \mathcal{B}_p^{s(\cdot,\cdot)}(u_p).
    \end{align*} 
    Therefore, $v_p \in \partial \mathcal{B}_p^{s(\cdot,\cdot)}(u_p)$, and 
    consequently $B_p^{s(\cdot,\cdot)} \subset \partial \mathcal{B}_p^{s(\cdot,\cdot)}$. Then, 
    since $B_p^{s(\cdot,\cdot)}$ is $m$-completely 
    accretive in $L^2(\Omega)$, we get $B_p^{s(\cdot,\cdot)} 
    = \partial \mathcal{B}_p^{s(\cdot,\cdot)}$.  
\end{proof}

The existence and uniqueness of weak solutions to problem \eqref{eqn:vfpl} is addressed 
in the following result.

\begin{theorem}
    For every $f \in L^2(\Omega)$ there exists a unique 
    weak solution of problem \eqref{eqn:vfpl} in $[0,T]$ for any $T>0$.
\end{theorem}

\begin{proof}
    By Theorem \ref{thm:mild-solution} and Theorem \ref{thm:vfpl} we have that, for every $f \in L^2(\Omega)$, there exists a 
    unique strong solution of the abstract Cauchy problem 
    \begin{equation}
        \left\{
        \begin{array}{ll}
            \displaystyle \vspace{0.25em} \frac{d u_p}{d t} + B_p^{s(\cdot,\cdot)}u_p \ni 0, & t \in (0,T), \\ 
            u(0) = f.
        \end{array}
        \right. \label{eqn:abc-vfpl}
    \end{equation}
    Now, the concept of weak solution of problem \eqref{eqn:vfpl} coincides with the 
    concept of strong solution of \eqref{eqn:abc-vfpl}, and the proof of the existence 
    and uniqueness concludes.
\end{proof}
    
\section{Proof of the main result} \label{5}

In this section, we will present the proof of our main result, namely the existence and 
uniqueness of weak solutions of problem \eqref{eqn:main}, following the approximation method 
in \cite{MAZON2016810}. To define the expression $\frac{u(y)-u(x)}{|u(y)-u(x)|}$, a multivalued 
sign $\mathrm{sign} (\cdot)$ will be used. The multivalued sign is defined 
by $\mathrm{sign} (r) = \frac{r}{|r|}$ if $r \ne 0$, and $[-1,1]$ if $r = 0$.

To study the problem \eqref{eqn:main} we consider the energy 
functional $\mathcal{B}_1^{s(\cdot,\cdot)}: L^2(\Omega) \to [0,+\infty]$ given by 
\[\mathcal{B}_1^{s(\cdot,\cdot)}(u)=\left\{
    \begin{array}{ll}
        \displaystyle \vspace{0.2em} \frac{1}{2} \int_{\Omega} \int_{\Omega} 
        \frac{k(x,y)\left|u(y)-u(x)\right|}{|x-y|^{N+s(x,y)}} 
        d x d y, & \textrm{if~~} u \in L^2(\Omega) \cap W^{s(\cdot,\cdot),1}(\Omega),  \\ 
        + \infty, & \textrm{if~~} u \in L^2(\Omega) \backslash W^{s(\cdot,\cdot),1}(\Omega).
    \end{array}
\right.\]

By Fatou's Lemma we have that $\mathcal{B}_1^{s(\cdot,\cdot)}$ is lower semi-continuous  
in $L^2(\Omega)$. Then, since $\mathcal{B}_1^{s(\cdot,\cdot)}$ is convex, we know 
that the subdifferential $\partial \mathcal{B}_1^{s(\cdot,\cdot)}$ is a maximal monotone operator 
in $L^2(\Omega)$. To characterize $\partial \mathcal{B}_1^{s(\cdot,\cdot)}$ we introduce the 
following operator:

\begin{definition}
    We define the operator $B_1^{s(\cdot,\cdot)}$ in $L^2(\Omega) \times 
    L^2(\Omega)$ by $v \in B_1^{s(\cdot,\cdot)}u$ if and only 
    if $u \in W^{s(\cdot,\cdot),1}(\Omega) \cap L^2(\Omega)$, $v \in L^2(\Omega)$ and 
    there exists a function $\eta \in L^{\infty}(\Omega \times \Omega)$, 
    $\eta(x,y)=-\eta(y,x)$ for almost all $(x,y) \in \Omega \times \Omega$, 
    $\|\eta\|_{L^{\infty}(\Omega \times \Omega)} \leq 1$ and 
    \[k(x,y)\eta(x,y) \in k(x,y) \mathrm{sign} (u(y)-u(x))\]
    for a.e. $(x,y) \in \Omega \times \Omega$, such that 
    \[\frac{1}{2} \int_{\Omega} \int_{\Omega} \frac{k(x,y)}{|x-y|^{N+s(x,y)}}
    \eta(x,y) (\varphi(y)-\varphi(x)) d x d y 
    =\int_{\Omega} v(x) \varphi(x) d x\]
    for all $\varphi \in W^{s(\cdot,\cdot),1}(\Omega) \cap L^2(\Omega)$.
\end{definition}

\begin{theorem} \label{thm:vf1l-ca}
    The operator $B_1^{s(\cdot,\cdot)}$ is completely accretive in $L^2(\Omega)$.
\end{theorem}

\begin{proof}
    Let $v_{i} \in B_1^{s(\cdot,\cdot)}u_{i}$, $i=1,2$, there exists 
    $\eta_i \in L^{\infty}(\Omega \times \Omega)$, 
    $\eta_i(x,y)=-\eta_i(y,x)$ for almost all $(x,y) \in \Omega \times \Omega$, 
    $\|\eta_i\|_{L^{\infty}(\Omega \times \Omega)} \leq 1$ and 
    \[k(x,y)\eta_i(x,y) \in k(x,y) \mathrm{sign} (u_i(y)-u_i(x))\]
    for a.e. $(x,y) \in \Omega \times \Omega$, such that 
    \[\frac{1}{2} \int_{\Omega} \int_{\Omega} \frac{k(x,y)}{|x-y|^{N+s(x,y)}}
    \eta_i(x,y) (\varphi(y)-\varphi(x)) d x d y 
    =\int_{\Omega} v_i(x) \varphi(x) d x\]
    for all $\varphi \in W^{s(\cdot,\cdot),1}(\Omega) \cap L^2(\Omega)$. 
    Given $q \in \mathbf{P}_0$ , taking $q(u_1(x)-u_2(x))$ as test function, we get
    \begin{align*} 
        &\int_{\Omega}\left(v_{1}(x)-v_{2}(x)\right) q\left(u_{1}(x)-u_{2}(x)\right) d x  
        \\  &=  \frac{1}{2} \int_{\Omega} \int_{\Omega} \frac{k(x,y)(\eta_{1}(x,y)-\eta_{2}(x,y))}{|x-y|^{N+s(x,y)}}  
        \left(q\left(u_{1}(y)-u_{2}(y)\right)-q\left(u_{1}(x)-u_{2}(x)\right)\right) d x d y   
        \\ &= \frac{1}{2} \iint\limits_{\{(x,y):u_1(y) \ne u_1(x), u_2(y)=u_2(x)\}} 
        \frac{k(x,y)}{|x-y|^{N+s(x,y)}} (\eta_{1}(x,y)-\eta_{2}(x,y)) \\
        & \quad \times \left(q\left(u_{1}(y)-u_{2}(y)\right)-q\left(u_{1}(x)-u_{2}(x)\right)\right) d x d y 
        \\ & \quad + \frac{1}{2} \iint\limits_{\{(x,y):u_1(y) = u_1(x), u_2(y) \ne u_2(x)\}} 
        \frac{k(x,y)}{|x-y|^{N+s(x,y)}} (\eta_{1}(x,y)-\eta_{2}(x,y)) \\
        & \quad \times \left(q\left(u_{1}(y)-u_{2}(y)\right)-q\left(u_{1}(x)-u_{2}(x)\right)\right) d x d y 
        \\ & \quad + \frac{1}{2} \iint\limits_{\{(x,y):u_1(y) \ne u_1(x), u_2(y) \ne u_2(x)\}} 
        \frac{k(x,y)}{|x-y|^{N+s(x,y)}} (\eta_{1}(x,y)-\eta_{2}(x,y)) \\
        & \quad \times \left(q\left(u_{1}(y)-u_{2}(y)\right)-q\left(u_{1}(x)-u_{2}(x)\right)\right) d x d y.   
    \end{align*} 
    The last three integrals are nonnegative since both $q$ and the multivalued sign are  
    monotonically non-decreasing. Hence, 
    \[\int_{\Omega} \left(v_{1}(x)-v_{2}(x)\right) q\left(u_{1}(x)-u_{2}(x)\right) dx \geq 0,\]
    from which it follows that $B_1^{s(\cdot,\cdot)}$ is completely accretive.
\end{proof}

\begin{theorem} \label{thm:vf1l}
    Let \eqref{eqn:trin} be satisfied, then the operator $B_1^{s(\cdot,\cdot)}$ is $m$-completely 
    accretive in $L^2(\Omega)$ with dense domain. Moreover, $B_1^{s(\cdot,\cdot)} 
    = \partial \mathcal{B}_1^{s(\cdot,\cdot)}$.
\end{theorem}

\begin{proof}
    Let us show that the operator $B_1^{s(\cdot,\cdot)}$ satisfies the range 
    condition $L^2(\Omega) \subset R \big( I + B_1^{s(\cdot,\cdot)} \big)$. From 
    now on $C$ denotes a positive constant independent of $p$, which can take different values 
    in different places. For $1<p<\frac{N}{N+s^+-1}$, take 
    \[s_p(x,y) \coloneqq \frac{N}{p_{s(x,y)}^*} = N-\frac{N}{p}+s(x,y),\]
    we have $s_p(x,y) = s_p(y,x)$ and 
    \[0<s^-<s_p^-= N-\frac{N}{p}+s^- \leq s_p(x,y)
    \leq s_p^+= N-\frac{N}{p}+s^+ < 1\]
    for all $(x,y) \in \Omega \times \Omega$. Then, given $f \in L^2(\Omega)$, 
    for $1<p<\frac{N}{N+s^+-1}$, applying Theorem \ref{thm:vfpl}, there exists 
    $u_p \in W^{s_p(\cdot,\cdot),p}(\Omega)$ such that $f-u_p \in B_p^{s_p(\cdot,\cdot)}u_p$. Now, 
    since $N+s_p(x,y)p=(N+s(x,y))p$, we have 
    \begin{equation}  \label{eqn:elliptic-vfpl-sp}
        \frac{1}{2} \int_{\Omega} \int_{\Omega} \frac{k(x,y)}{|x-y|^{(N+s(x,y))p}} U_p(x,y)
    (\varphi(y)-\varphi(x)) d x d y =\int_{\Omega} (f(x)-u_p(x)) \varphi(x) d x
    \end{equation}
    for all $\varphi \in W^{s_p(\cdot,\cdot),p}(\Omega) \cap L^2(\Omega)$. Moreover, since 
    $B_p^{s_p(\cdot,\cdot)}$ is completely accretive and $0 \in B_p^{s_p(\cdot,\cdot)}(0)$, it 
    is easy to see that $u_p \ll f$, i.e. for all $j \in \mathbf{J}_0$, 
    \[\int_{\Omega} j(u_p) dx \leq \int_{\Omega} j(f) dx,\]
    from which we deduce that $\|u_p\|_{L^2(\Omega)} \leq \|f\|_{L^2(\Omega)}$. 
    Therefore, there exists a sequence $p_n \searrow 1$ such that 
    $u_{p_n} \rightharpoonup u$ weakly in $L^2(\Omega)$, 
    and $\|u\|_{L^2(\Omega)} \leq \|f\|_{L^2(\Omega)}$. 
    On the other hand, taking $\varphi=u_p$ in \eqref{eqn:elliptic-vfpl-sp} we have
    \begin{equation}  \label{eqn:elliptic-vfpl-phiequp}
        \frac{1}{2} \int_{\Omega} \int_{\Omega} \frac{k(x,y)}{|x-y|^{(N+s(x,y))p}} |u_p(y)-u_p(x)|^p
     d x d y =\int_{\Omega} (f(x)-u_p(x)) u_p(x) d x \leq C_0.
    \end{equation}
    By Lemma \ref{lem:bounded} we know that 
    \begin{align*}
        \int_{\Omega} \int_{\Omega} \frac{|u_p(y)-u_p(x)|^p}{|x-y|^{(N+s(x,y))p}} dxdy 
      &\leq \frac{C}{C_1} + \frac{4\big(\|f\|_{L^2(\Omega)}^2+1\big) \big(|\Omega|^2+1\big)}{\min\big\{1,\delta_1^{N+2}\big\}} \\
      &= C\big(C_1,C_0,\delta_1,N,|\Omega|,\|f\|_{L^2(\Omega)}\big).
    \end{align*}
    Applying Hölder's inequality, we have 
    \begin{align*}
        [u_p]_{W^{s(\cdot,\cdot),1}(\Omega)}
      &=\int_{\Omega} \int_{\Omega} \frac{|u_p(y)-u_p(x)|}{|x-y|^{N+s(x,y)}} d x d y \\
      &\leq \left(\int_{\Omega} \int_{\Omega} \frac{|u_p(y)-u_p(x)|^p}{|x-y|^{(N+s(x,y))p}} dxdy\right)^{\frac{1}{p}} |\Omega \times \Omega|^{\frac{p-1}{p}} \\ 
      &\leq (C+1) \big(|\Omega|^2+1\big).
    \end{align*}
    Thus, $\|u_p\|_{W^{s(\cdot,\cdot),1}(\Omega)} \leq C$. By Theorem 
    \ref{thm:embedding}, we have that for a subsequence of $\{p_n\}$, which will not be 
    relabeled, $u_{p_n} \to u$ strongly in $L^1(\Omega)$ and $u \in W^{s(\cdot,\cdot),1}(\Omega)$. 
    
    Now, for any $1< \sigma < \frac{p}{p-1}$, using Hölder's inequality, 
    \begin{align*} 
        &\left\| \left| \frac{u_{p}(y)-u_{p}(x)}{|x-y|^{N+s(x,y)}}\right|^{p-2} \frac{u_{p}(y)-u_{p}(x)}{|x-y|^{N+s(x,y)}} \right\|_{L^{\sigma}(\Omega \times \Omega)}^{\sigma} \\  
    &\leq \left\| \left| \frac{u_{p}(y)-u_{p}(x)}{|x-y|^{N+s(x,y)}}\right|^{p-2} \frac{u_{p}(y)-u_{p}(x)}{|x-y|^{N+s(x,y)}} \right\|_{L^{\frac{p}{p-1} }(\Omega \times \Omega)}^{\sigma} |\Omega \times \Omega|^{1-\frac{p-1}{p} \sigma} \\ 
    &= \left(\int_{\Omega} \int_{\Omega} \frac{\left|u_{p}(y)-u_{p}(x)\right|^p}{|x-y|^{N+s_p(x,y) p}} d x d y \right)^{\frac{p-1}{p} \sigma} |\Omega \times \Omega|^{1-\frac{p-1}{p} \sigma} \\ 
    &\leq [u_p]_{W^{s_p(\cdot,\cdot),p}(\Omega)}^{\frac{p-1}{p} \sigma} |\Omega \times \Omega|^{1-\frac{p-1}{p} \sigma}\\ 
    &\leq (C+1)\big(|\Omega|^2+1\big). 
    \end{align*}
    By a diagonal argument, it is clear that there exists a subsequence of $\{p_n\}$, denoted equal, 
    \[\left| \frac{u_{p_n}(y)-u_{p_n}(x)}{|x-y|^{N+s(x,y)}}\right|^{p_n-2} \frac{u_{p_n}(y)-u_{p_n}(x)}{|x-y|^{N+s(x,y)}} \rightharpoonup \eta(x,y)\]
    weakly in $L^{\sigma}(\Omega \times \Omega)$, with $\eta$ antisymmetric 
    such that $\eta \in L^{\sigma}(\Omega \times \Omega)$ for all $1< \sigma < + \infty$.  
    Moreover, 
    \[\| \eta \|_{L^{\sigma}(\Omega \times \Omega)}^{\sigma} 
    \leq \underset{n \to \infty}{\underline{\text{lim}}} 
    \left\| \left| \frac{u_{p_{n}}(y)-u_{p_{n}}(x)}{|x-y|^{N+s(x,y)}}\right|^{p_n-2} 
    \frac{u_{p_{n}}(y)-u_{p_{n}}(x)}{|x-y|^{N+s(x,y)}} 
    \right\|_{L^{\sigma}(\Omega \times \Omega)}^{\sigma} \leq C,\]
    then we have $\eta \in L^{\infty}(\Omega \times \Omega)$ and 
    $\| \eta \|_{L^{\sigma}(\Omega \times \Omega)} \leq C^{\frac{1}{\sigma}}$. Letting 
    $\sigma \to + \infty$, we get $\| \eta \|_{L^{\infty}(\Omega \times \Omega)} \leq 1$.  
    Now let us pass to the limit in \eqref{eqn:elliptic-vfpl-sp}. Let us first 
    take $\varphi \in C^{\infty}(\overline{\Omega})$. Since 
    \begin{align*} 
        &\lim_{n \to \infty} \frac{1}{2} \int_{\Omega} \int_{\Omega} \frac{k(x,y)}{|x-y|^{N+s_{p_n}(x,y) p_n}}  U_{p_n}(x,y) (\varphi(y)-\varphi(x)) dxdy \\
        &=\lim_{n \to \infty} \frac{1}{2} \int_{\Omega} \int_{\Omega} \frac{k(x,y)}{|x-y|^{N+s(x,y)}} \left| \frac{u_{p_n}(y)-u_{p_n}(x)}{|x-y|^{N+s(x,y)}}\right|^{p_n-2} \frac{u_{p_n}(y)-u_{p_n}(x)}{|x-y|^{N+s(x,y)}} 
        \times (\varphi(y)-\varphi(x)) dxdy\\ 
        &=\frac{1}{2} \int_{\Omega} \int_{\Omega}  \frac{k(x,y)}{|x-y|^{N+s(x,y)}} \eta(x,y) (\varphi(y)-\varphi(x)) dxdy 
    \end{align*}
    we have
    \begin{equation}  \label{eqn:elliptic-vf1l}
        \frac{1}{2} \int_{\Omega} \int_{\Omega}  \frac{k(x,y)}{|x-y|^{N+s(x,y)}} \eta(x,y)
         (\varphi(y)-\varphi(x)) dxdy =\int_{\Omega} \left(f(x)-u(x)\right) \varphi(x) d x.
    \end{equation}
    Suppose now $\varphi \in W^{s(\cdot,\cdot),1}(\Omega) \cap L^2(\Omega)$. By 
    Theorem \ref{thm:density} we know that there exists $\varphi_n \in C^{\infty}(\overline{\Omega})$ such 
    that $\varphi_n \to \varphi$ in $L^2(\Omega)$ and 
    $[\varphi_n - \varphi]_{W^{s(\cdot,\cdot),1}(\Omega)} \to 0$ as $n \to \infty$. Therefore,  
    \begin{align*}
        &\left|\int_{\Omega} \int_{\Omega} \frac{k(x,y)\left(\varphi_{n}(y)-\varphi_{n}(x)\right)}{|x-y|^{N+s(x,y)}} dxdy 
        - \int_{\Omega} \int_{\Omega} \frac{k(x,y)\left(\varphi(y)-\varphi(x)\right)}{|x-y|^{N+s(x,y)}} dxdy\right| \\ 
        &\leq C_2 \int_{\Omega} \int_{\Omega} \frac{\left|(\varphi_{n}(y)-\varphi(y))-(\varphi_{n}(x)-\varphi(x))\right|}{|x-y|^{N+s(x,y)}} dxdy \\ 
        &= C_2 [\varphi_n - \varphi]_{W^{s(\cdot,\cdot),1}(\Omega)} \to 0
    \end{align*}
    as $n \to \infty$. By Fatou's Lemma and \eqref{eqn:elliptic-vf1l}, we have 
    \begin{align*}
        &\frac{1}{2} \int_{\Omega} \int_{\Omega}\left(\frac{k(x,y)|\varphi(y)-\varphi(x)|}{|x-y|^{N+s(x,y)}}-\frac{k(x,y)}{|x-y|^{N+s(x,y)}} \eta(x, y)(\varphi(y)-\varphi(x))\right) dxdy \\ 
        &\leq \underset{n \to \infty}{\underline{\text{lim}}} \frac{1}{2} \int_{\Omega} \int_{\Omega}\left(\frac{k(x,y)\left|\varphi_{n}(y)-\varphi_{n}(x)\right|}{|x-y|^{N+s(x,y)}}
        -\frac{k(x,y)\eta(x, y)\left(\varphi_{n}(y)-\varphi_{n}(x)\right)}{|x-y|^{N+s(x,y)}} \right) dxdy \\ 
        &= \frac{1}{2} \int_{\Omega} \int_{\Omega} \frac{k(x,y)}{|x-y|^{N+s(x,y)}}|\varphi(y)-\varphi(x)| dxdy -\int_{\Omega}(f(x)-u(x)) \varphi(x) dx, 
    \end{align*}
    which implies 
    \[\frac{1}{2} \int_{\Omega} \int_{\Omega}  \frac{k(x,y)}{|x-y|^{N+s(x,y)}} \eta(x,y)
    (\varphi(y)-\varphi(x)) dxdy - \int_{\Omega} \left(f(x)-u(x)\right) \varphi(x) d x \geq 0\]
    for all $\varphi \in W^{s(\cdot,\cdot),1}(\Omega) \cap L^2(\Omega)$, and hence, since 
    the above inequality is also true for $-\varphi$, we get the equality 
    \begin{equation}  \label{eqn:elliptic-vf1l-2}
        \frac{1}{2} \int_{\Omega} \int_{\Omega}  \frac{k(x,y)}{|x-y|^{N+s(x,y)}} \eta(x,y)
    (\varphi(y)-\varphi(x)) dxdy - \int_{\Omega} \left(f(x)-u(x)\right) \varphi(x) d x = 0
    \end{equation}
    for all $\varphi \in W^{s(\cdot,\cdot),1}(\Omega) \cap L^2(\Omega)$.

    To finish the proof of the range condition, we only need to show that 
    \[k(x,y)\eta(x,y) \in k(x,y) \mathrm{sign} (u(y)-u(x))\]
    for a.e. $(x,y) \in \Omega \times \Omega$. By \eqref{eqn:elliptic-vfpl-phiequp} 
    for $p_n$ and taking $\varphi=u$ in \eqref{eqn:elliptic-vf1l-2}, we have 
    \begin{align*}
        &\frac{1}{2} \int_{\Omega} \int_{\Omega}  \frac{k(x,y)}{|x-y|^{(N+s(x,y)) p_{n}}}\left|u_{p_{n}}(y)-u_{p_{n}}(x)\right|^{p_{n}} d x d y \\ 
        &= \int_{\Omega}\left(f(x)-u_{p_{n}}(x)\right) u_{p_{n}}(x) d x \\  
        &=  \int_{\Omega}(f(x)-u(x)) u(x) d x-\int_{\Omega} f(x)\left(u(x)-u_{p_{n}} (x)\right) d x \\
        &\quad+2 \int_{\Omega} u(x)\left(u(x)-u_{p_{n}}(x)\right) d x-\int_{\Omega}(u(x)-u_{p_{n}}(x))^{2} d x \\ 
        &\leq \frac{1}{2} \int_{\Omega} \int_{\Omega} \frac{k(x,y)}{|x-y|^{N+s(x,y)}} \eta(x, y)(u(y)-u(x)) d x d y 
        -\int_{\Omega} f(x)\left(u(x)-u_{p_{n}}(x)\right) d x \\ 
        &\quad+2 \int_{\Omega} u(x)\left(u(x)-u_{p_{n}}(x)\right) d x,
    \end{align*} 
    Then, taking limit as $n \to \infty$, we get
    \begin{equation*} 
        \underset{n \to \infty}{\overline{\text{lim}}} 
    \frac{1}{2}\int_{\Omega} \int_{\Omega}  \frac{k(x,y)}{|x-y|^{(N+s(x,y)) p_{n}}}\left|u_{p_{n}}(y)-u_{p_{n}}(x)\right|^{p_{n}} dxdy 
     \leq \frac{1}{2}\int_{\Omega} \int_{\Omega}  \frac{k(x,y)}{|x-y|^{N+s(x,y)}} \eta(x,y) (u(y) - u(x)) dxdy,
    \end{equation*} 
    By the lower semi-continuity in $L^1(\Omega)$ of $[\cdot]_{W^{s(\cdot,\cdot),1}}$, 
    we have 
    \begin{align*}
        &\frac{1}{2} \int_{\Omega} \int_{\Omega} \frac{k(x,y)}{|x-y|^{N+s(x,y)}}|u(y)-u(x)| dxdy \\ 
        &\leq \underset{n \to \infty}{\underline{\text{lim}}} \frac{1}{2} \int_{\Omega} \int_{\Omega} \frac{k(x,y)}{|x-y|^{N+s(x,y)}}\left|u_{p_{n}}(y)-u_{p_{n}}(x)\right| dxdy \\ 
        &\leq\underset{n \to \infty}{\underline{\text{lim}}} \frac{1}{2} \left(\int_{\Omega} \int_{\Omega} \frac{k(x,y)}{|x-y|^{(N+s(x,y)) p_{n}}}\left|u_{p_{n}}(y)-u_{p_{n}}(x)\right|^{p_{n}} dxdy\right)^{\frac{1}{p_{n}}}
        \left(C_2|\Omega \times \Omega|\right)^{\frac{p_{n}-1}{p_{n}}}\\  
        &\leq\frac{1}{2} \int_{\Omega} \int_{\Omega} \frac{k(x,y)}{|x-y|^{N+s(x,y)}} \eta(x, y)(u(y)-u(x)) dxdy,
    \end{align*}
    from which we obtain that 
    \[k(x,y)\eta(x,y) \in k(x,y) \mathrm{sign} (u(y)-u(x))\]
    for a.e. $(x,y) \in \Omega \times \Omega$. Then, we 
    have $f-u \in B_1^{s(\cdot,\cdot)}u$. Hence, by Lemma \ref{thm:vf1l-ca} we 
    obtain that $B_1^{s(\cdot,\cdot)}$ is $m$-completely accretive in $L^2(\Omega)$.

    Let us now see that $\mathrm{Dom}(B_1^{s(\cdot,\cdot)})$ is dense $L^2(\Omega)$. To 
    see this fact it is enough to show that 
    \[C^{\infty}(\overline{\Omega}) \subset
     \overline{\mathrm{Dom}\big(B_1^{s(\cdot,\cdot)}\big)} ^ {L^{2}(\Omega)}.\]
    Given $v \in C^{\infty}(\overline{\Omega})$, since $B_1^{s(\cdot,\cdot)}$ 
    is $m$-completely accretive in $L^2(\Omega)$, there 
    exists $u_n \in \mathrm{Dom}\big(B_1^{s(\cdot,\cdot)}\big)$ such that 
    $n(v-u_n) \in B_1^{s(\cdot,\cdot)}u_n$. Hence, there exists 
    $\eta_n \in L^{\infty}(\Omega \times \Omega)$, 
    $\eta_n(x,y)=-\eta_n(y,x)$ for almost all $(x,y) \in \Omega \times \Omega$, 
    $\|\eta_n\| \leq 1$, such that 
    \[\frac{1}{2} \int_{\Omega} \int_{\Omega}  \frac{k(x,y)}{|x-y|^{N+s(x,y)}} \eta_n(x,y)
    (\varphi(y)-\varphi(x)) dxdy = n \int_{\Omega} \left(v(x)-u_n(x)\right) \varphi(x) d x\]
    for all $\varphi \in W^{s(\cdot,\cdot),1}(\Omega) \cap L^2(\Omega)$, and 
    \[k(x,y)\eta_n(x,y) \in k(x,y) \mathrm{sign} (u_n(y)-u_n(x))\]
    for a.e. $(x,y) \in \Omega \times \Omega$. Then, taking $\varphi = v - u_n$, 
    we have 
    \begin{align*} 
       & \int_{\Omega} (v(x)-u_{n}(x))^2 d x \\
       &= \frac{1}{2n} \int_{\Omega} \int_{\Omega} \frac{k(x,y)}{|x-y|^{N+s(x,y)}} \eta_n(x,y) 
       (v(y)-u_n(y)-(v(x)-u_n(x))) dxdy \\
       &\leq \frac{1}{2n} \int_{\Omega} \int_{\Omega} \frac{k(x,y)}{|x-y|^{N+s(x,y)}}  
       |v(y)-v(x)| dxdy \\
       &\leq \frac{C_2}{2n} [v]_{W^{s(\cdot,\cdot), 1}(\Omega)},
    \end{align*}
    from which it follows that $u_n \to v$ in $L^2(\Omega)$.

    Finally, let us show that $B_1^{s(\cdot,\cdot)} = \partial \mathcal{B}_1^{s(\cdot,\cdot)}$. 
    Let $v \in B_1^{s(\cdot,\cdot)} u$, there exists $\eta \in L^{\infty}(\Omega \times \Omega)$, 
    $\eta(x,y)=-\eta(y,x)$ for almost all $(x,y) \in \Omega \times \Omega$, 
    $\|\eta\| \leq 1$, such that
    \[\frac{1}{2} \int_{\Omega} \int_{\Omega}  \frac{k(x,y)}{|x-y|^{N+s(x,y)}} \eta(x,y)
    (\varphi(y)-\varphi(x)) dxdy = \int_{\Omega} v(x) \varphi(x) d x,\]
    and 
    \[k(x,y)\eta(x,y) \in k(x,y) \mathrm{sign} (u(y)-u(x))\]
    for a.e. $(x,y) \in \Omega \times \Omega$. Then, 
    given $w \in W^{s(\cdot,\cdot),1}(\Omega) \cap L^2(\Omega)$, 
    we have 
    \begin{align*} 
     \int_{\Omega} v(x) (w(x)-u(x)) d x 
    &= \frac{1}{2} \int_{\Omega} \int_{\Omega} \frac{k(x,y)}{|x-y|^{N+s(x,y)}} \eta(x,y) (w(y)-w(x)) dxdy - \mathcal{B}_{1}^{s(\cdot,\cdot)}(u) \\ 
    &\leq  \mathcal{B}_{1}^{s(\cdot,\cdot)}(w) - \mathcal{B}_{1}^{s(\cdot,\cdot)}(u).
    \end{align*}
    Therefore, $v \in \partial \mathcal{B}_1^{s(\cdot,\cdot)} (u)$, and 
    consequently $B_1^{s(\cdot,\cdot)} \subset \partial \mathcal{B}_1^{s(\cdot,\cdot)}$. Then, 
    since $B_1^{s(\cdot,\cdot)}$ is $m$-completely 
    accretive in $L^2(\Omega)$, we get $B_1^{s(\cdot,\cdot)} 
    = \partial \mathcal{B}_1^{s(\cdot,\cdot)}$.
\end{proof}

Working as in the proof of Theorem \ref{thm:vfpl}, we get the following result about existence and 
uniqueness of weak solutions to problem \eqref{eqn:main}.

\begin{proof}[\bf Proof of Theorem \ref{thm:main}]
    By Theorem \ref{thm:mild-solution} and Theorem \ref{thm:vf1l} we have that, for every $f \in L^2(\Omega)$, there exists a 
    unique strong solution of the abstract Cauchy problem 
    \begin{equation}
        \left\{
        \begin{array}{ll}
            \displaystyle \vspace{0.25em} \frac{d u}{d t} + B_1^{s(\cdot,\cdot)}u \ni 0, & t \in (0,T), \\ 
            u(0) = f.
        \end{array}
        \right. \label{eqn:abc-vf1l}
    \end{equation}
    Now, the concept of weak solution of problem \eqref{eqn:main} coincides with the 
    concept of strong solution of \eqref{eqn:abc-vf1l}, and the proof of the existence 
    and uniqueness concludes.
\end{proof}

\section{Some properties of the proposed model} \label{6}

In this section, we investigate some properties of weak solutions of problem \eqref{eqn:main}. 
Firstly, let us build the extremum principle as follows.

\begin{theorem} \label{thm:extremum-principle}
    Suppose $f \in L^2(\Omega)$ with $\underset{x \in \Omega}{\mathrm{inf}}~ f > 0$ and 
    $\underset{x \in \Omega}{\mathrm{sup}}~ f < + \infty$. Let $u$ be the weak solution of 
    problem \eqref{eqn:main} with the initial data $f$, then for every $t>0$,
    \begin{equation}
        \inf_{x \in \Omega} f \leq u \leq \sup_{x \in \Omega} f, 
        \quad \mathrm{a.e.} \ x \in \Omega. \label{eqn:extremum}
    \end{equation}
\end{theorem}

\begin{proof} 
    Let $I=\underset{x \in \Omega}{\mathrm{inf}}~ f$, $S=\underset{x \in \Omega}{\mathrm{sup}}~ f$. 
    Denote $u_-=\min\{u-I,0\}$, taking $\varphi=u_-$ in \eqref{eqn:main-weak}, we know that 
    there exists a function $\eta(t,\cdot,\cdot) \in L^{\infty}(\Omega \times \Omega)$, 
    $\eta(t,x,y)=-\eta(t,y,x)$ for almost all $(x,y) \in \Omega \times \Omega$, 
    $\|\eta(t,\cdot,\cdot)\|_{L^{\infty}(\Omega \times \Omega)} \leq 1$ and 
    $k(x,y)\eta(t,x,y) \in k(x,y) \mathrm{sign} (u(t,y)-u(t,x))$
    for a.e. $(x,y) \in \Omega \times \Omega$, such that 
    \begin{equation}
        \frac{1}{2} \frac{d}{dt} \int_{\Omega} |u_-(x,t)|^2 d x
        +\frac{1}{2} \int_{\Omega} \int_{\Omega} \frac{k(x,y)}{|x-y|^{N+s(x,y)}}
        \eta(t,x,y)(u_-(y)-u_-(x)) d x d y =0, \label{eqn:main-weak-mp}
    \end{equation}
    for almost all $t>0$. For the second term on the left hand of \eqref{eqn:main-weak-mp}, 
    we have 
    \begin{align*}
        &\int_{\Omega} \int_{\Omega} \frac{k(x,y)}{|x-y|^{N+s(x,y)}}
        \eta(t,x,y)(u_-(y)-u_-(x)) d x d y \\  
        & = \int_{\{y:u(y)<I\}} \int_{\{x:u(x)<I\}} 
        \frac{k(x,y)}{|x-y|^{N+s(x,y)}} \eta(t,x,y)(u(y)-u(x)) d x d y \\
        & \quad +2 \int_{\{y:u(y)<I\}} \int_{\{x:u(x) \geq I\}} 
        \frac{k(x,y)} {|x-y|^{N+s(x,y)}}\eta(t,x,y)(u(y)-I) d x d y\\ 
        & \geq 0
    \end{align*}
    for almost all $t>0$. Then 
    \[\frac{1}{2} \frac{d}{dt} \int_{\Omega} |u_-(x,t)|^2 d x \leq 0\] 
    for all $t>0$ due to the smoothness of $t \mapsto u(t)$ for $t>0$. 
    Therefore, $\displaystyle \frac{1}{2} \frac{d}{dt} \int_{\Omega} |u_-(x,t)|^2 d x$ is 
    decreasing in $t$, and since for all $t>0$,
    \[0 \leq \int_{\Omega} |u_-(x,t)|^2 d x \leq \int_{\Omega} |u_-(x,0)|^2 d x = 0,\] 
    we have that
    \[\int_{\Omega} |u_-(x,t)|^2 d x = 0\] 
    for all $t>0$, and so $u(t) \geq I$ a.e. on $\Omega$, $\forall t > 0$. Denote 
    $u_+=\max\{u-S,0\}$, a similar argument yields 
    that $u(t) \leq S$ a.e. on $\Omega$, $\forall t > 0$. The 
    extremum principle \eqref{eqn:extremum} is followed directly.   
\end{proof}

By utilizing the properties of completely accretive operators, it is straightforward 
to obtain the following contraction principle and comparison principle.

\begin{theorem} \label{thm:contraction-comparison}
    Suppose $f_1, \ f_2 \in L^2(\Omega)$ and $u_1, \ u_2$ are weak solutions of 
    problem \eqref{eqn:main} with the initial data $f_1, \ f_2$, then for every $t>0$,
    \begin{equation}
        \int_{\Omega} (u_1(t)-u_2(t))^+ dx \leq \int_{\Omega} (f_1-f_2)^+ dx. \label{eqn:contraction}
    \end{equation} 
    Moreover, if $f_1 \leq f_2$, then $u_1 \leq u_2$ a.e. on $\Omega$.
\end{theorem}

\begin{proof} 
    Since $B_1^{s(\cdot,\cdot)}$ is completely accretive in $L^2(\Omega)$, it is 
    also $T$-completely accretive in $L^2(\Omega)$ (see \cite[Definition A.55]{ANDREU2010}). Thus, the resolvents 
    of $B_1^{s(\cdot,\cdot)}$ are $T$-contractions, which implies the 
    contraction principle \eqref{eqn:contraction}.

    In the case of $f_1 \leq f_2$, it is clear that $(f_1-f_2)^+ = 0$. By 
    the contraction principle we obtain that 
    \[\int_{\Omega} (u_1(t)-u_2(t))^+ dx = 0,\]
    which implies $(u_1-u_2)^+ = 0$ for every $t>0$ and 
    a.e. $x \in \Omega$, i.e. $u_1 \leq u_2$ a.e. on $\Omega$.
\end{proof}

The following result implies the stability of the weak solutions.

\begin{theorem} \label{thm:stability}
    Assume $u(t)$ is the weak solution of problem \eqref{eqn:main} with the initial 
    data $f \in L^2(\Omega)$, then for every $t>0$,
    \[\int_{\Omega} u(t) dx = \int_{\Omega} f dx.\]
\end{theorem}

\begin{proof} 
    Taking $\varphi=1$ as test function in \eqref{eqn:main-weak}, we obtain that for 
    every $t \geq 0$, 
    \[\int_{\Omega} \frac{\partial u(t,x)}{\partial t} d x = 0,\]
    from which it follows that the function $t \mapsto \displaystyle \int_{\Omega} u(t,x)dx$ is 
    constant, thus 
    \[\int_{\Omega} u(t,x) dx = \int_{\Omega} u(0,x) dx = \int_{\Omega} f dx\] 
    for every $t>0$.
\end{proof}

At the end of this section, we present a result that reveals the asymptotic behavior 
of the weak solutions as time tends to infinity.

\begin{theorem} \label{thm:asymptotic-behavior}
    Assume $u(t)$ is the weak solution of problem \eqref{eqn:main} with the initial 
    data $f \in L^2(\Omega)$, then when $t \to + \infty$, $u$ converges 
    in the $L^1(\Omega)$-strong topology to the mean value of the initial 
    data $f_{\Omega}$, i.e.
    \[\lim_{t \to + \infty} \int \left|u(t,x)-f_{\Omega}\right| dx = 0.\]
\end{theorem}

\begin{proof} 
    Taking $u(t)$ as test function in \eqref{eqn:main-weak}, we get
    \[-\frac{1}{2} \frac{d}{dt} \int_{\Omega} |u(x,t)|^2 d x
    =\frac{1}{2} \int_{\Omega} \int_{\Omega} \frac{k(x,y)}{|x-y|^{N+s(x,y)}}
    |u(y)-u(x)| d x d y\]
    for every $t \geq 0$, which implies 
    \[\int_{0}^{t}[u(\tau)]_{W^{s(\cdot,\cdot),1}(\Omega)} d \tau 
    \leq 2 \|f\|_{L^2(\Omega)}^2\]
    for every $t \geq 0$. Then, using Theorem \ref{thm:poincare} and 
    Theorem \ref{thm:stability}, we have 
    \[\int_{\Omega} \left|u(t,x)-f_{\Omega}\right| dx 
    =\int_{\Omega} \left|u(t,x)-u_{\Omega}\right| dx 
    \leq C [u(t,\cdot)]_{W^{s(\cdot,\cdot),1}}.\]

    Since $B_1^{s(\cdot,\cdot)}$ is completely accretive 
    and $0 \in B_1^{s(\cdot,\cdot)}(0)$, we have 
    \[\int_{\Omega} |u(t,x)-f_{\Omega}| dx 
    \leq \int_{\Omega} |u(\tau,x)-f_{\Omega}| dx\]
    if $t \geq \tau$. Then, for every $t > 0$, 
    \begin{align*}\int_{\Omega} |u(t,x)-f_{\Omega}| dx 
    &\leq \frac{1}{t} \int_{0}^{t} \int_{\Omega} |u(\tau,x)-f_{\Omega}| dx d \tau
    \leq \frac{C}{t} \int_{0}^{t}[u(\tau)]_{W^{s(\cdot,\cdot),1}(\Omega)} d \tau \\
    &\leq \frac{2C}{t} \|f\|_{L^2(\Omega)}^2. \end{align*}
    Let $t \to + \infty$, we obtain that 
    \[\lim_{t \to + \infty} \int \left|u(t,x)-f_{\Omega}\right| dx = 0,\]
    which concludes the proof.
\end{proof}

\section{Numerical experiments} \label{7}

In this section, we demonstrate the effectiveness of the proposed model in denoising 
with the help of several numerical examples. Firstly, we focus on the selection 
of appropriate functions $k(x,y)$ and $s(x,y)$.

To remove multiplicative noise effectively, it is necessary to 
select the appropriate weighted function $k(x,y)$. We rewrite $k(x,y)=a(x,y)b(x,y)$. $a(x,y)$ can 
be set to $a(x,y) \equiv  1$ or a gray value detection function 
$a(x,y) = \left(\frac{f_{\sigma}(x)f_{\sigma}(y)}{M^2}\right)^r$, where $f_{\sigma} = G_{\sigma} * f$, 
$M = \underset{x \in \Omega}{\mathrm{max}}~f_{\sigma}(x)$,
$G_{\sigma}(x) = \frac{1}{2\pi \sigma^2} \exp \big(-\frac{|x|^2}{2 \sigma^2}\big)$, $\sigma, r$ are 
given constants. $b(x,y)$ can be set as the boundary detection function in F1P-AA model, i.e.
\[b(x,y) =  \exp \left(- \frac{|g(x)-g(y)|^2}{h_g}\right) \chi_{\{|x-y| \leq \eta\}},\]
where $g = G_{\sigma_g} * f $, 
$G_{\sigma_g}(x) = \frac{1}{2\pi \sigma_g^2} \exp \big(-\frac{|x|^2}{2 \sigma_g^2}\big)$, 
$h_g, \sigma_g$ and $\eta$ are some given constants. If we 
select $a(x,y) = \left(\frac{f_{\sigma}(x)f_{\sigma}(y)}{M^2}\right)^r$, it is clear 
that $0 < a(x,y) \leq 1$ and $a(x,y)$ increases with the gray 
value of the smoothed noisy image $f_{\sigma}$. When the gray value is low, 
$a(x,y)$ becomes small which leads to slow 
diffusion at low gray value regions. When the gray value is high, 
$a(x,y)$ tends to $1$ and then the boundary 
detection function $b(x,y)$ will play the main role in the diffusion 
process. $b(x,y)$ becomes small when $|g(x)-g(y)|$ is large, which leads to 
the protection of edges. When $|g(x)-g(y)|$ is small, the diffusion will 
be fast and noise in flat regions will be removed efficiently.

In order for our model to be practically useful, it is fundamental that we determine 
the function $s(x,y)$. We first use a set of Gabor filters to extract the 
texture information from the noisy image, i.e. $h = H * f$,
where $H(x) = \sum_{k=1}^K H_k(x; \lambda_k, \theta_k, \sigma_k)$,  
$H_k(x; \lambda_k, \theta_k, \sigma_k) 
= \exp \left(-\frac{|x|^2}{2 \sigma_k^2}\right) 
\cos \left(\frac{2 \pi}{\lambda_k} b_k^{\mathrm{T}} x\right)$, 
$b_k = (\cos \theta_k, \sin \theta_k)^{\mathrm{T}}$. Then we choose the function $s(x,y)$ as   
\[s(x,y) = s^- + (s^+ - s^-) \exp \left(- m |h(x)-h(y)|^2\right),\]
where $\sigma_k, K, \lambda_k, \theta_k, s^-, s^+, m$ are some given constants. 
Gabor filters are widely used in texture analysis, as they have been found to 
be particularly appropriate for representing textures in specific directions 
in a localized region, see \cite{MANJUNATH1996837} and the references therein. 
A set of Gabor filters with different 
scales and with orientations in different directions can be used 
to obtain the texture feature $h$ of the input noisy image $f$. 
Precisely, $s(x,y) \in (0,1)$ is a symmetric function, $|h(x)-h(y)|^2$ can 
be regard as the proximity between the gray value $h(x)$ and $h(y)$. 
We can set $s^+$ as a value very close to $1$, when $|h(x)-h(y)|$ is small, 
$s(x,y)$ tends to $s^+$, which leads to an approximation of integer-order diffusion, 
removing noise in texture-poor regions. On the other hand, we 
can set $s^-$ as a non-extreme order, when $|h(x)-h(y)|$ is large, 
$s(x,y)$ tends to $s^-$, resulting in typical nonlocal diffusion, 
which preserves textures in texture-rich regions. 

Now we present the finite difference scheme for our model. Assume that $\tau$ 
to be the time step size and $h$ the space grid size. We take $\tau = 0.5$, $h = 1$ in 
practical calculation. The time and space discretizations are presented as follows:
\[\begin{array}{ll}
    t = n \tau, & n = 0,1,2,\cdots \\
    x = ih, & y = jh, \quad i = 1,2,\cdots,I, \quad j = 1,2,\cdots,J,
\end{array}\]
where $Ih \times Jh$ is the size of the original image. Denote by $u_{i,j}^n$ the 
approximation of $u(n \tau, ih, jh)$. Then the numerical 
approximation of problem \eqref{eqn:main} could be written as
\begin{align*}
    &u_{i,j}^{n+1} = u_{i,j}^{n} + \tau \sum_{\substack{
i-\eta \leq i' \leq i+\eta, \\ 
j-\eta \leq j' \leq j+\eta, \\
(i',j') \ne (i,j)}} \frac{k((i'h,j'h),(ih,jh))
\left(\mathrm{sign}_0(u_{i',j'}^{n} - u_{i,j}^{n})\right)}
{\left(h\sqrt{(i'-i)^2+(j'-j)^2}\right)^{2+s((i'h,j'h),(ih,jh))}} , \\
    &u_{i,j}^{0} = f_{i,j} = f(ih,jh), \quad 0 \leq i \leq I, \ 
    0 \leq j \leq J. 
\end{align*}

Through the above lines, we can obtain $u_{i,j}^1$ by $u_{i,j}^0$. The 
restoration quality is measured by the peak signal noise ratio (PSNR) and 
the structural similarity index measure (SSIM), which are defined as 
\[\begin{array}{c}
    \mathrm{PSNR} = 10 \log_{10} \left(\dfrac{\Sigma_{i,j} 255^2}
    {\Sigma_{i,j} (u_{i,j} - \bar{u}_{i,j})^2}\right), \\[2ex] 
    \mathrm{SSIM}(x,y) = \dfrac{(2\mu_x \mu_y +c_1)(2 \sigma_{xy} + c_2)}
    {\mu_x^2 + \mu_y^2 + c_1 + \sigma_x^2 + \sigma_y^2 + c_2},
\end{array}\]
where $\bar{u}$ is the original image and $u$ denotes the compared image, $c_1, c_2$ 
are two variables to stabilize the division with weak denominator, 
$\mu_x, \mu_y, \sigma_x, \sigma_y$ and $\sigma_{xy}$ are the local means, 
standard deviations and cross-covariance for image $x,y$, respectively. The 
better quality image will have higher values of PSNR and SSIM. Our algorithm 
stop when it achieves its maximal PSNR. 

To verify the effectiveness of our model, we test several images which 
are distorted by multiplicative Gamma noise with mean $1$ and $L \in \{1,4,10\}$. 
For illustration, the results for the 256-by-256 gray value texture1, texture2, hybrid, and 
the 512-by-512 gray value satellite1, satellite2 are presented, see the original test 
images in Figure \ref{original}. The results are compared with SO model 
\cite{SHI2008294}, AA model \cite{AUBERT2008925} and F1P-AA model \cite{GAO20224837}. 
The PSNR and SSIM values of different $L$ are list in Table \ref{results}.
 Parameters 
$\sigma, r, \sigma_g, h_g, \eta, s^-$ and $s^+$ are set as $1, 0.6, 15, 
10, 3, 0.5$ and $ 0.99$, respectively. The 
parameters $\lambda_k, \theta_k, \sigma_k$ of Gabor filters are 
selected following the method presented in \cite{MANJUNATH1996837}. We use 
$4$ orientations and $8$ scales, yielding a total of $K=32$ Gabor filters for each 
experiment. The parameter $m$ is set as $1$ or $\log_2 \left(\frac{2L+4}{3}\right)$.

\begin{figure}[htbp] 
    \centering
    \begin{minipage}[b]{.16\linewidth}
        \centering
        \includegraphics[width=\textwidth]{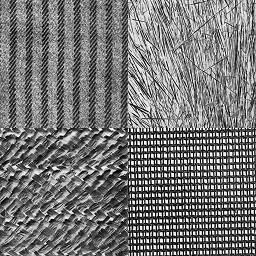}
        \subcaption{Texture1}
    \end{minipage} 
    \begin{minipage}[b]{.16\linewidth}
        \centering
        \includegraphics[width=\textwidth]{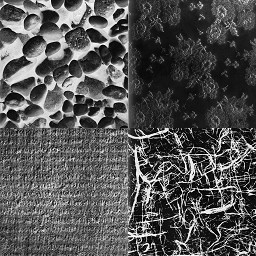}
        \subcaption{Texture2}
    \end{minipage}
    \begin{minipage}[b]{.16\linewidth}
        \centering
        \includegraphics[width=\textwidth]{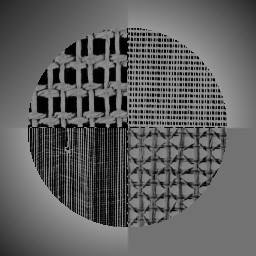}
        \subcaption{Hybrid}
    \end{minipage}
    \begin{minipage}[b]{.16\linewidth}
        \centering
        \includegraphics[width=\textwidth]{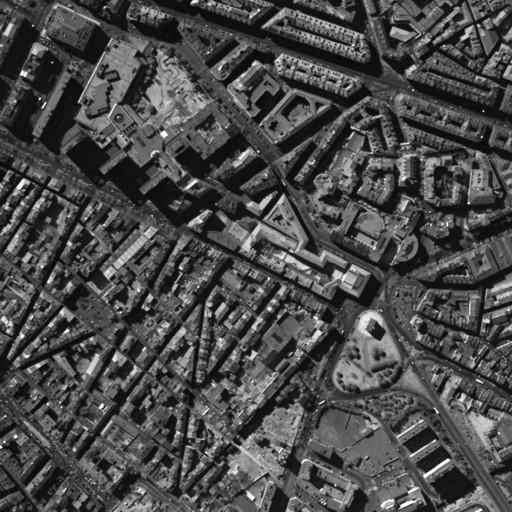}
        \subcaption{Satellite1}
    \end{minipage}
    \begin{minipage}[b]{.16\linewidth}
        \centering
        \includegraphics[width=\textwidth]{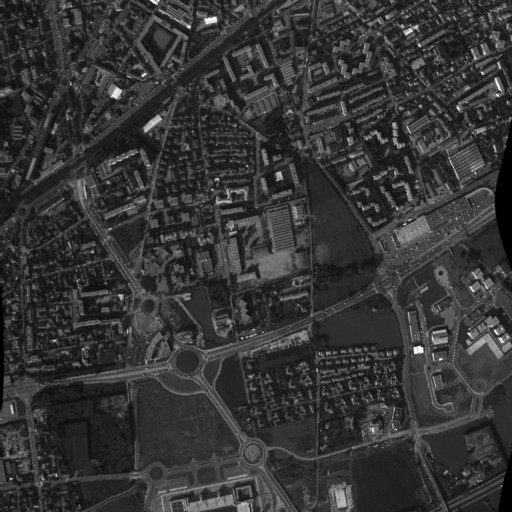}
        \subcaption{Satellite2}
    \end{minipage}
    \caption{Test images.} \label{original}
\end{figure}

\begin{table}[htbp] 
    \centering
    \caption{Comparison of PSNR and SSIM of the different models for gamma noises with $L=1,4,10$.}
    \label{results}
    \scalebox{0.9}{\begin{tabular}{c*{10}c}
        \toprule
        \multirow{2}*{} & \multicolumn{6}{c}{PSNR} & \multicolumn{2}{c}{SSIM} \\
        \cmidrule{3-10}
          &      & Ours  & SO  & AA  & F1P-AA & Ours  & SO  & AA  & F1P-AA \\
        \midrule
        \multirow{5}*{$L=1$} 
        & Texture1 & \textbf{12.54} & 11.07 & \textbf{12.54} & \textbf{12.54} & 0.4289 & 0.4127 & \textbf{0.4357} & 0.4288 \\
        & Texture2 & \textbf{14.71} & 13.29 & 13.36 & 14.70 & \textbf{0.4488} & 0.4232 & 0.4298 & 0.4484 \\
        & Hybrid & \textbf{18.34} & 14.41 & 16.78 & 18.33 & 0.5256 & 0.3394 & 0.5082 & \textbf{0.5258} \\
        & Satellite1 & 18.87 & 16.00 & \textbf{19.18} & 18.83 & \textbf{0.7091} & 0.6284 & 0.6985 & 0.6811 \\
        & Satellite2 & 22.74 & 19.10 & \textbf{23.21} & 22.70 & 0.6025 & 0.5520 & \textbf{0.6327} & 0.6000 \\
        \cmidrule{1-10}
        \multirow{5}*{$L=4$} & Texture1 & \textbf{15.68} & 15.35 & 15.52 & 15.51 & 0.6885 & \textbf{0.6896} & 0.6845 & 0.6813 \\
        & Texture2 & \textbf{17.96} & 17.73 & 17.88 & 17.67 & \textbf{0.6933} & 0.6912 & 0.6917 & 0.6831 \\
        & Hybrid & \textbf{21.19} & 20.28 & 19.16 & 20.44 & \textbf{0.6990} & 0.5467 & 0.6129 & 0.6309 \\
        & Satellite1 & \textbf{22.16} & 21.59 & 22.13 & 21.58 & 0.8681 & \textbf{0.8697} & 0.8567 & 0.8513 \\
        & Satellite2 & \textbf{25.42} & 24.71 & 25.39 & 24.73 & \textbf{0.8229} & 0.8204 & 0.8156 & 0.8056 \\
        \cmidrule{1-10}
        \multirow{5}*{$L=10$} & Texture1 & \textbf{18.23} & 18.03 & 18.03 & 18.05 & \textbf{0.8170} & 0.8047 & 0.8141 & 0.8129 \\
        & Texture2 & \textbf{20.43} & 20.23 & 20.10 & 20.26 & \textbf{0.8133} & 0.8054 & 0.8076 & 0.8111 \\
        & Hybrid & \textbf{23.61} & 23.42 & 22.71 & 22.62 & \textbf{0.7985} & 0.6082 & 0.7211 & 0.7229 \\
        & Satellite1 & \textbf{24.35} & 24.27 & 24.22 & 23.65 & 0.9280 & \textbf{0.9327} & 0.9230 & 0.9205 \\ 
        & Satellite2 & \textbf{27.47} & 27.23 & 27.30 & 26.91 & \textbf{0.9010} & 0.9008 & 0.8998 & 0.8959 \\ 
        \bottomrule
    \end{tabular}}
\end{table}

Now we report the numerical experiments of multiplicative noise removal for the 
original test images in Figure \ref{original}. All test images are corrupted by 
multiplicative Gamma noise with $L=4$ and the corresponding results are depicted 
in Figures \ref{texture1}-\ref{satellite2}. Figures \ref{texture1} and \ref{texture2} 
show the comparison of experimental results for two texture images. We find that Figures 
\ref{texture1-SO}, \ref{texture2-SO} and 
\ref{texture1-AA}, \ref{texture2-AA} lost more texture information 
in the images than Figures \ref{texture1-Ours}, \ref{texture2-Ours} and 
\ref{texture1-F1P-AA}, \ref{texture2-F1P-AA}. Contrariwise, our model exhibits 
comparable performance to the F1P-AA model in removing noise, while surpassing it 
in preserving texture information. In the experiment which is shown in 
Figure \ref{hybrid}, we utilize a test image that consists of textures and smooth surfaces. 
It is observed from Figures \ref{hybrid-SO} and \ref{hybrid-AA} that SO cause residual 
noise in homogeneous regions and AA fails to preserve texture information. Furthermore, 
our model performs better in preserving texture compared to the F1P-AA model, and also 
has favorable denoising effect across the entire image, avoiding the drawbacks of each 
model, see Figure \ref{hybrid-Ours} and Figure \ref{hybrid-F1P-AA}. As illustrated 
in Figures \ref{satellite1} and \ref{satellite2}, the denoising 
results of the AA and F1P-AA model exhibit a certain degree of oversmoothing, while 
the results of the SO contain more residual noise. In comparison, although there are 
some isolated white speckle points in Figure \ref{satellite1-Ours} and 
Figure \ref{satellite2-Ours}, our model achieves a balance between denoising 
effectiveness and texture preservation.

\begin{figure}[t]
    \centering
    \begin{minipage}[b]{.16\linewidth}
        \centering
        \includegraphics[width=\textwidth]{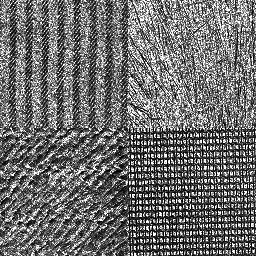}
        \subcaption{Noisy}
    \end{minipage} 
    \hspace{0.2em}
    \begin{minipage}[b]{.16\linewidth}
        \centering
        \includegraphics[width=\textwidth]{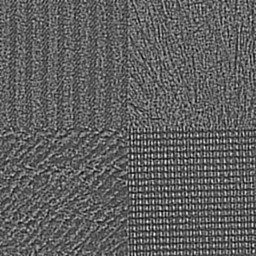}
        \subcaption{Gabor}
    \end{minipage}
    \hspace{0.2em}
    \begin{minipage}[b]{.16\linewidth}
        \centering
        \includegraphics[width=\textwidth]{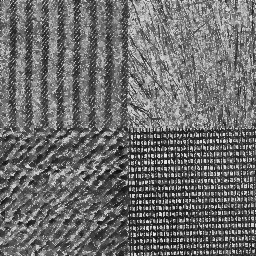}
        \subcaption{Ours} \label{texture1-Ours}
    \end{minipage}
    
    \begin{minipage}[b]{.16\linewidth}
        \centering
        \includegraphics[width=\textwidth]{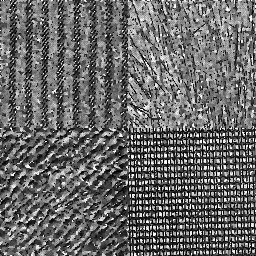}
        \subcaption{SO} \label{texture1-SO}
    \end{minipage}
    \hspace{0.2em}
    \begin{minipage}[b]{.16\linewidth}
        \centering
        \includegraphics[width=\textwidth]{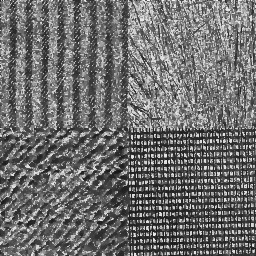}
        \subcaption{AA} \label{texture1-AA}
    \end{minipage}
    \hspace{0.2em}
    \begin{minipage}[b]{.16\linewidth}
        \centering
        \includegraphics[width=\textwidth]{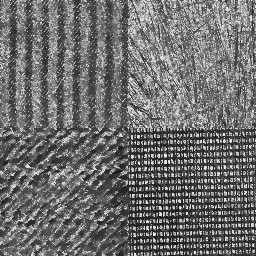}
        \subcaption{F1P-AA} \label{texture1-F1P-AA}
    \end{minipage}
    \caption{The texture1 image: (a) Noisy: $L=4$, PSNR $=13.7636$. 
    (b) Gabor filtered noisy image. (c)-(f) Denoised results.} \label{texture1}
\end{figure}

\begin{figure}[t]
    \centering
    \begin{minipage}[b]{.16\linewidth}
        \centering
        \includegraphics[width=\textwidth]{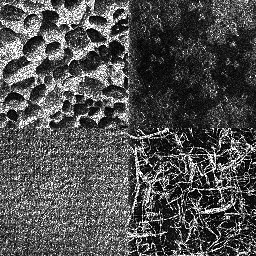}
        \subcaption{Noisy}
    \end{minipage} 
    \hspace{0.2em}
    \begin{minipage}[b]{.16\linewidth}
        \centering
        \includegraphics[width=\textwidth]{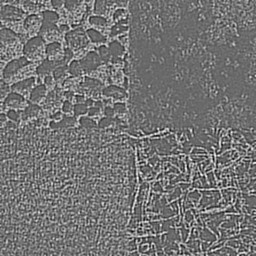}
        \subcaption{Gabor}
    \end{minipage}
    \hspace{0.2em}
    \begin{minipage}[b]{.16\linewidth}
        \centering
        \includegraphics[width=\textwidth]{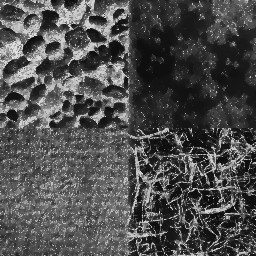}
        \subcaption{Ours} \label{texture2-Ours}
    \end{minipage}
    
    \begin{minipage}[b]{.16\linewidth}
        \centering
        \includegraphics[width=\textwidth]{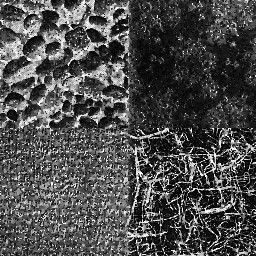}
        \subcaption{SO} \label{texture2-SO}
    \end{minipage}
    \hspace{0.2em}
    \begin{minipage}[b]{.16\linewidth}
        \centering
        \includegraphics[width=\textwidth]{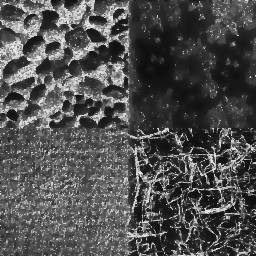}
        \subcaption{AA} \label{texture2-AA}
    \end{minipage}
    \hspace{0.2em}
    \begin{minipage}[b]{.16\linewidth}
        \centering
        \includegraphics[width=\textwidth]{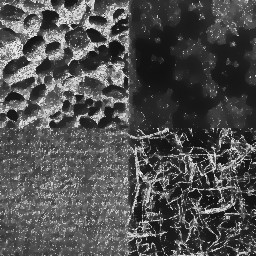}
        \subcaption{F1P-AA} \label{texture2-F1P-AA}
    \end{minipage}
    \caption{The texture2 image: (a) Noisy: $L=4$, PSNR $=15.5765$. 
    (b) Gabor filtered noisy image. (c)-(f) Denoised results.} \label{texture2}
\end{figure}

\begin{figure}[t]
    \centering
    \begin{minipage}[b]{.16\linewidth}
        \centering
        \includegraphics[width=\textwidth]{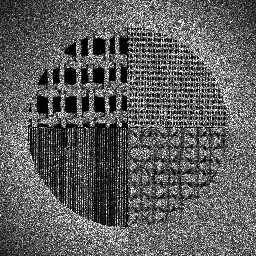}
        \subcaption{Noisy}
    \end{minipage} 
    \hspace{0.2em}
    \begin{minipage}[b]{.16\linewidth}
        \centering
        \includegraphics[width=\textwidth]{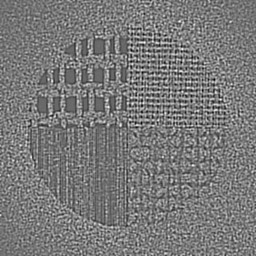}
        \subcaption{Gabor}
    \end{minipage}
    \hspace{0.2em}
    \begin{minipage}[b]{.16\linewidth}
        \centering
        \includegraphics[width=\textwidth]{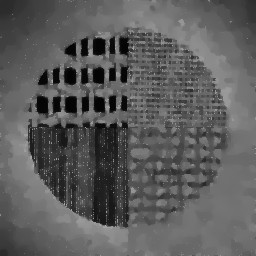}
        \subcaption{Ours} \label{hybrid-Ours}
    \end{minipage}
    
    \begin{minipage}[b]{.16\linewidth}
        \centering
        \includegraphics[width=\textwidth]{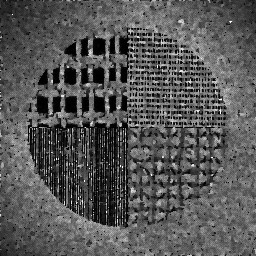}
        \subcaption{SO} \label{hybrid-SO}
    \end{minipage}
    \hspace{0.2em}
    \begin{minipage}[b]{.16\linewidth}
        \centering
        \includegraphics[width=\textwidth]{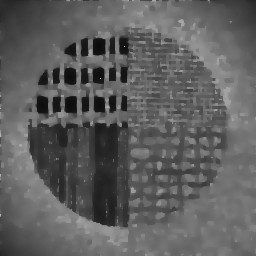}
        \subcaption{AA} \label{hybrid-AA}
    \end{minipage}
    \hspace{0.2em}
    \begin{minipage}[b]{.16\linewidth}
        \centering
        \includegraphics[width=\textwidth]{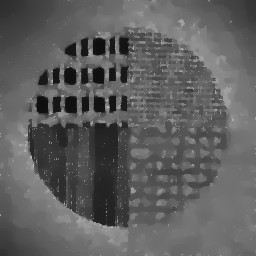}
        \subcaption{F1P-AA} \label{hybrid-F1P-AA}
    \end{minipage}
    
    \caption{The hybrid image: (a) Noisy: $L=4$, PSNR $=14.3859$. 
    (b) Gabor filtered noisy image. (c)-(f) Denoised results.} \label{hybrid} 
\end{figure}

\begin{figure}[t]
    \centering
    \begin{minipage}[b]{.16\linewidth}
        \centering
        \includegraphics[width=\textwidth]{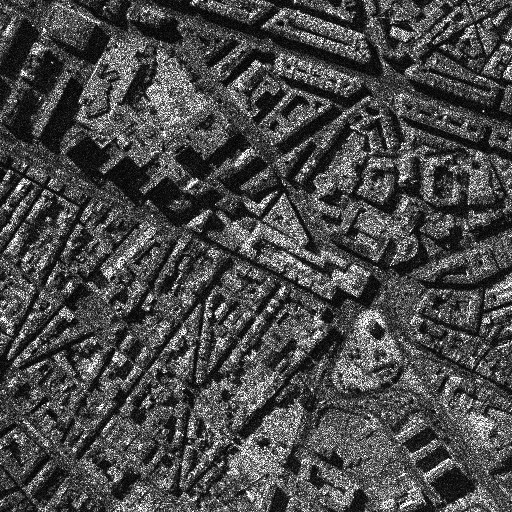}
        \subcaption{Noisy}
    \end{minipage} 
    \hspace{0.2em}
    \begin{minipage}[b]{.16\linewidth}
        \centering
        \includegraphics[width=\textwidth]{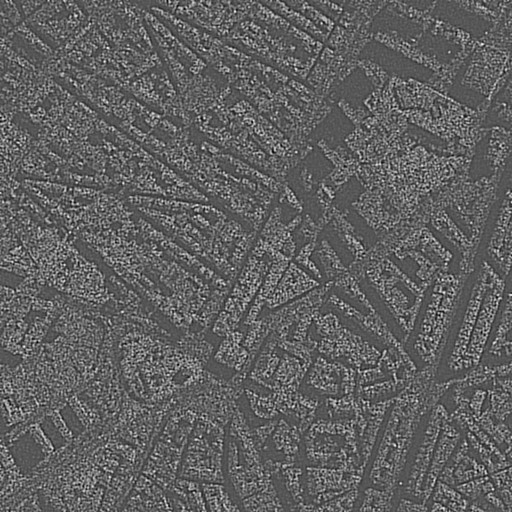}
        \subcaption{Gabor}
    \end{minipage}
    \hspace{0.2em}
    \begin{minipage}[b]{.16\linewidth}
        \centering
        \includegraphics[width=\textwidth]{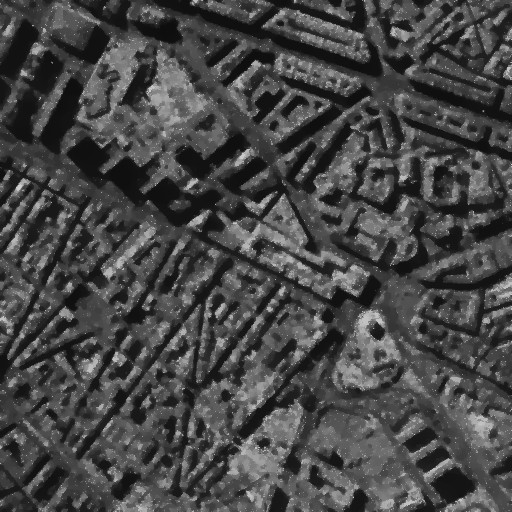}
        \subcaption{Ours} \label{satellite1-Ours}
    \end{minipage}
    
    \begin{minipage}[b]{.16\linewidth}
        \centering
        \includegraphics[width=\textwidth]{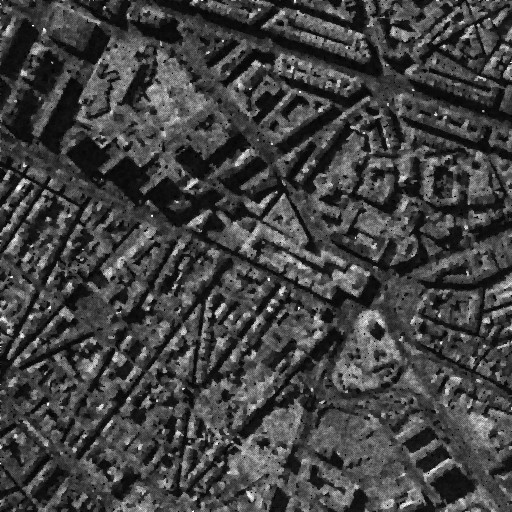}
        \subcaption{SO}
    \end{minipage}
    \hspace{0.2em}
    \begin{minipage}[b]{.16\linewidth}
        \centering
        \includegraphics[width=\textwidth]{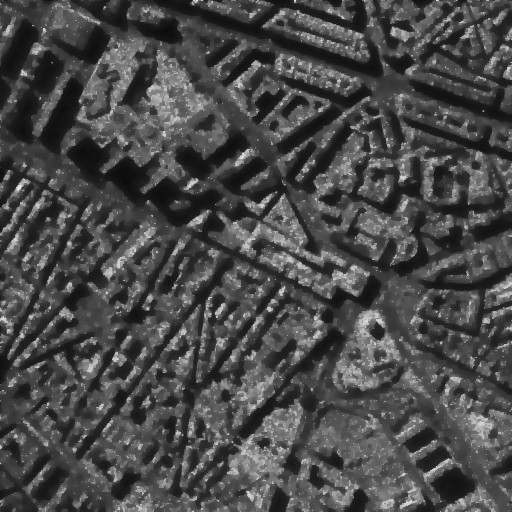}
        \subcaption{AA}
    \end{minipage}
    \hspace{0.2em}
    \begin{minipage}[b]{.16\linewidth}
        \centering
        \includegraphics[width=\textwidth]{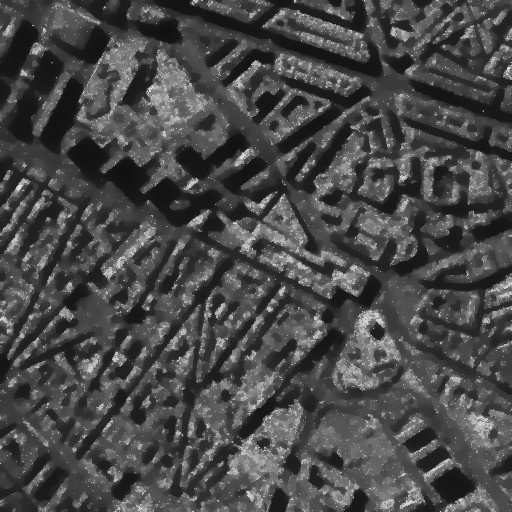}
        \subcaption{F1P-AA}
    \end{minipage}
    \caption{The satellite1 image: (a) Noisy: $L=4$, PSNR $=16.1315$. 
    (b) Gabor filtered noisy image. (c)-(f) Denoised results.} \label{satellite1}
\end{figure}

\begin{figure}[t]
    \centering
    \begin{minipage}[b]{.16\linewidth}
        \centering
        \includegraphics[width=\textwidth]{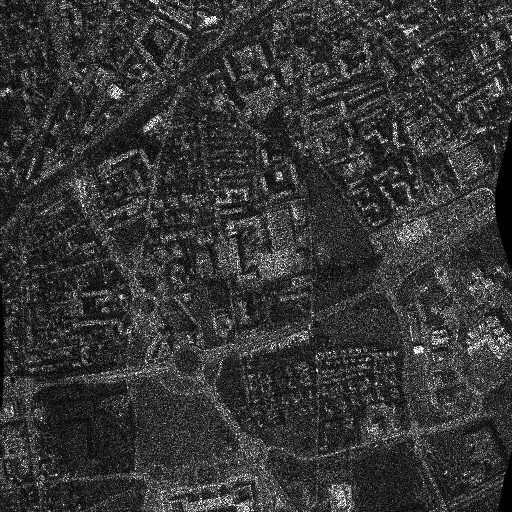}
        \subcaption{Noisy}
    \end{minipage} 
    \hspace{0.2em}
    \begin{minipage}[b]{.16\linewidth}
        \centering
        \includegraphics[width=\textwidth]{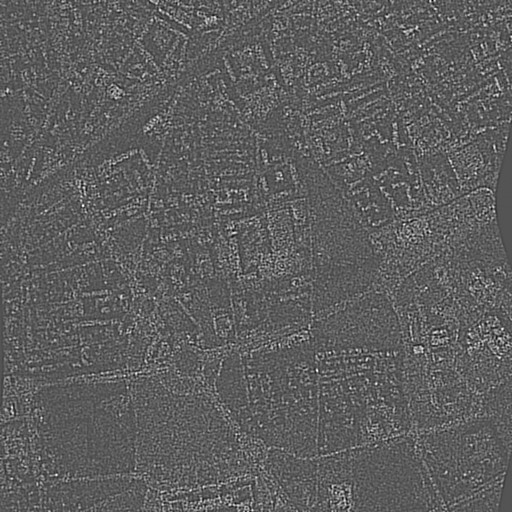}
        \subcaption{Gabor}
    \end{minipage}
    \hspace{0.2em}
    \begin{minipage}[b]{.16\linewidth}
        \centering
        \includegraphics[width=\textwidth]{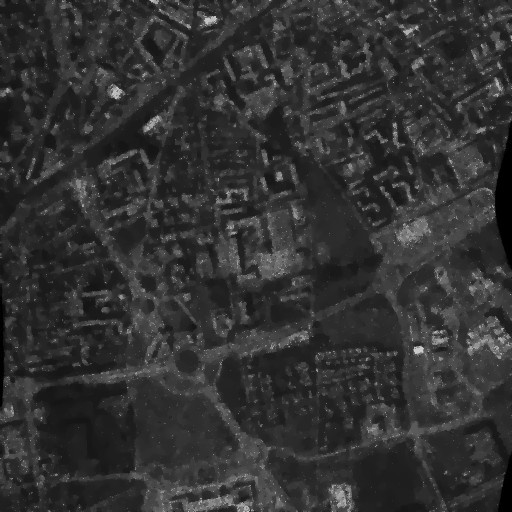}
        \subcaption{Ours} \label{satellite2-Ours}
    \end{minipage}
    
    \begin{minipage}[b]{.16\linewidth}
        \centering
        \includegraphics[width=\textwidth]{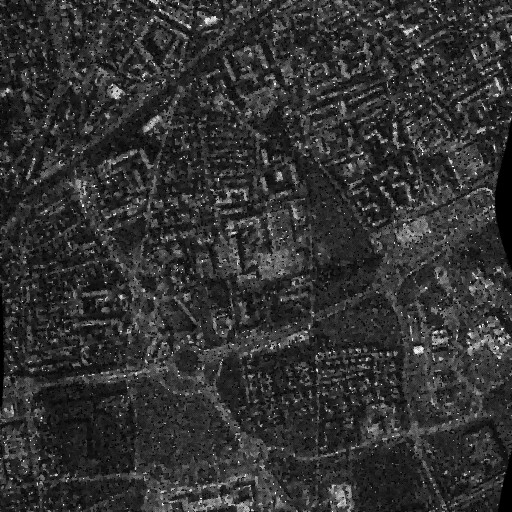}
        \subcaption{SO}
    \end{minipage}
    \hspace{0.2em}
    \begin{minipage}[b]{.16\linewidth}
        \centering
        \includegraphics[width=\textwidth]{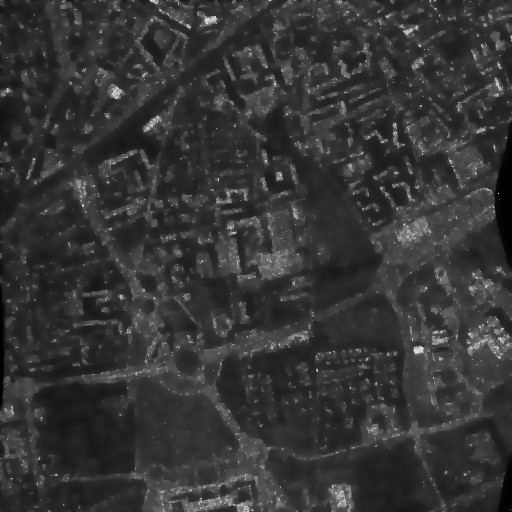}
        \subcaption{AA}
    \end{minipage}
    \hspace{0.2em}
    \begin{minipage}[b]{.16\linewidth}
        \centering
        \includegraphics[width=\textwidth]{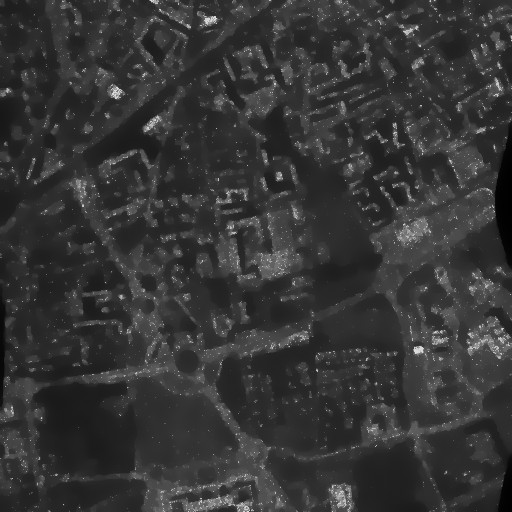}
        \subcaption{F1P-AA}
    \end{minipage}
    \caption{The satellite2 image: (a) Noisy: $L=4$, PSNR $=15.5765$. 
    (b) Gabor filtered noisy image. (c)-(f) Denoised results.} \label{satellite2}
\end{figure}

\bibliographystyle{elsarticle-names}
\bibliography{vf1p}

\end{document}